\documentclass[10pt]{amsart}
\usepackage[utf8]{inputenc}
\usepackage{a4wide}
\usepackage{amsmath, amssymb, bbm}
\usepackage{amsthm}
\swapnumbers
\usepackage{mathrsfs}
\usepackage{bbm}
\usepackage{blkarray}
\usepackage[matrix,arrow,curve]{xy}
\usepackage{graphicx}
\usepackage{color}
\usepackage{enumerate}
\definecolor{darkblue}{rgb}{0,0,0.6}
\usepackage[ocgcolorlinks,colorlinks=true, citecolor=darkblue, filecolor=darkblue, linkcolor=darkblue, urlcolor=darkblue]{hyperref}
\usepackage[capitalize,noabbrev]{cleveref}
\usepackage{comment}
\usepackage{pifont}

\newcounter{commentcounter}

		\newcommand{\commento}[1]
	{\stepcounter{commentcounter}
		 \textbf{Comment \arabic{commentcounter}}:
	\textcolor{red}{#1} }

\newcommand\ML[1]{\marginpar{\tiny ML: #1}}
\newcommand\fh[1]{\marginpar{\tiny FH: #1}}

\newtheorem{Prop}{Proposition}[subsection]
\newtheorem{Thm}[Prop]{Theorem}
\newtheorem{Lemma}[Prop]{Lemma}
\newtheorem{Cor}[Prop]{Corollary}

\newtheorem*{Thm*}{Theorem}
\newtheorem*{Prop*}{Proposition}

\newtheorem*{corollary*}{Corollary}
\newtheorem*{maintheorem}{Main theorem}

\theoremstyle{name}

\theoremstyle{defs2}

\newtheorem{conj}[Prop]{Conjecture}
\newtheorem*{conjbbc*}{Block Borel conjecture}
\newtheorem*{conjbc*}{Central part of Burghelea's conjecture}
\newtheorem*{conj*}{Conjecture}
\newtheorem*{mconj*}{Main Conjecture}

\theoremstyle{definition}
\newtheorem{Def}[Prop]{Definition}
\newtheorem*{ackn}{Acknowledgements}

\theoremstyle{remark}
\newtheorem{Rmk}[Prop]{Remark}
\newtheorem{Question}{Question}

\theoremstyle{claim}
\newtheorem*{claim}{Claim}

\newcommand{\G}{\mathcal{G}}
\newcommand{\CP}{\mathcal{CP}}
\newcommand{\CBC}{\mathcal{CBP}}
\newcommand{\KC}{\mathcal{KP}}
\newcommand{\IKP}{\mathcal{IKP}}
\newcommand{\GC}{\mathcal{VP}}
\newcommand{\IVP}{\mathcal{IVP}}
\newcommand{\BC}{\mathcal{BP}}

\newcommand{\Ep}{\widetilde E^+}
\newcommand{\R}{\mathbb{R}}
\newcommand{\wt}{\widetilde}
\newcommand{\Ext}{\mathrm{Ext}}
\newcommand{\B}{\mathrm{B}}

\newcommand{\TOP}{\mathrm{Top}}
\newcommand{\DIFF}{\mathrm{Diff}}
\newcommand{\TOPtw}{\wt{\TOP}}
\newcommand{\DIFFtw}{\wt{\DIFF}}

\newcommand{\VC}{\mathrm{vc}}
\newcommand{\LFJ}{\mathcal{L}\calfj^\fib_{\VC}[\tfrac{1}{2}]}
\newcommand{\LFJfin}{\mathcal{L}\calfj^\fib_{\mathrm{fin}}[\tfrac{1}{2}]}
\newcommand{\BSO}{\B\mathrm{SO}}

\newcommand{\CAT}{\mathrm{Cat}}
\newcommand{\PL}{\mathrm{PL}}
\newcommand{\BSTOP}{\B\mathrm S\TOP}
\newcommand{\BTOPtw}{\B\TOPtw}

\newcommand{\Sym}{\mathrm{Sym}}
\newcommand{\Gd}{\mu_\Gamma}
\newcommand{\xto}{\xrightarrow}
\newcommand{\lto}{\longrightarrow}

\newcommand{\Q}{\mathbb{Q}}
\newcommand{\Z}{\mathbb{Z}}
\newcommand{\Hom}{\mathrm{Hom}}

\newcommand{\C}{C}

\newcommand{\cd}{\mathrm{cd}}

\newcommand{\tr}{\mathrm{tr}}
\newcommand{\Out}{\mathrm{Out}}

\newcommand{\s}{\mathcal{S}}

\newcommand{\id}{\mathrm{id}}

\newcommand{\ab}{\mathrm{ab}}
\newcommand{\calfj}{\mathcal{F}\!\mathcal{J}}

\newcommand{\fib}{\mathrm{fib}}
\newcommand{\pbc}{identity block Borel conjecture}

\newcommand{\e}{e^{fw}}

\newcommand{\Sp}{\mathrm{Sp}}
\newcommand{\Ss}{\mathbb S}
\newcommand{\fin}{\mathrm{fin}}

\DeclareMathOperator*{\colim}{colim}
\DeclareMathOperator*{\hocolim}{hocolim}

\newcommand\lra{\longrightarrow}
\newcommand{\bR}{\mathbb{R}}

\keywords{aspherical closed manifolds, tautological classes}
\subjclass[2010]{55R20, 55R40, 55R60, 57P10}

\title{A Vanishing theorem for tautological classes of aspherical manifolds}
\date{\today}

\author[F.~Hebestreit]{Fabian Hebestreit}
\address{University of Bonn, Mathematical Institute, Endenicher Allee 60, 53115 Bonn, Germany}
\email{f.hebestreit@math.uni-bonn.de}

\author[M.~Land]{Markus Land}
\address{University of Regensburg, NWF I -- Mathematik, 93049 Regensburg, Germany}
\email{markus.land@mathematik.uni-regensburg.de}

\author[W.~L\"uck]{Wolfgang L\"uck}
\address{University of Bonn, Mathematical Institute, Endenicher Allee 60, 53115 Bonn, Germany}
\email{wolfgang.lueck@him.uni-bonn.de}

\author[O.~Randal-Williams]{Oscar Randal-Williams}
\address{University of Cambridge, DPMMS, Wilberforce Road, Cambridge CB3 0WB, UK}
\email{o.randal-williams@dpmms.cam.ac.uk}

\begin{document}

\begin{abstract}
Tautological classes, or generalised Miller--Morita--Mumford classes, are basic characteristic classes of smooth fibre bundles, and have recently been used to describe the rational cohomology of classifying spaces of diffeomorphism groups for several types of manifolds. We show that rationally tautological classes depend only on the underlying topological block bundle, and use this to prove the vanishing of tautological classes for many bundles with fibre an aspherical manifold.
\end{abstract}

\maketitle
\setcounter{tocdepth}{1}
\tableofcontents

\typeout{------------------- Introduction -----------------}
\section{Introduction}
\label{sec:introduction}

Spaces of automorphisms of manifolds have long been an active topic of research in topology, and various techniques have emerged for their study. In the case of high-dimensional manifolds, there are two competing approaches: On the one hand, one tries to understand the difference between the space of diffeomorphisms and the space of homotopy self-equivalences by introducing yet another space, the space of block diffeomorphisms, whose difference to homotopy equivalences is measured by surgery theory and whose difference to diffeomorphisms is measured, at least in a range depending only on the dimension of the manifold, in terms of Waldhausen's $A$-theory; see \cite{WW2} for a modern approach. An example of this approach being successfully employed is \cite{FH}, where Farrell and Hsiang investigate the rational homotopy type of various spaces of automorphisms, and in particular determine the rational homotopy groups of the space of homeomorphisms of aspherical manifolds in a range. This has a recent integral refinement in \cite{ELPUW}.

On the other hand, with the work of Madsen, Tillmann, and Weiss on Mumford's conjecture, a new line of investigation emerged. This approach is based on cobordism theory and tries to describe the cohomology of the classifying space of diffeomorphisms in terms of a certain Thom spectrum -- an object accessible to the computational methods of algebraic topology. This method is particularly well suited to studying specific cohomology classes, the generalised Miller--Morita--Mumford classes. Since they are central to the present article let us briefly recall their definition.

Given a smooth, oriented fibre bundle $p\colon E \rightarrow B$ with typical fibre a compact, closed, oriented manifold $M$ of dimension $d$, a coefficient ring $R$, and a characteristic class $c \in H^k(\BSO(d);R)$, the associated \emph{Miller--Morita--Mumford class}, or \emph{tautological class}, is the cohomology class
\[\kappa_c(p) = p_! (c(T_v(p))) \in H^{k-d}(B;R)\]
obtained by applying the Gysin homomorphism $p_!$ associated to $p$ to the class $c(T_v(p)) \in H^k(E;R)$ given by evaluating the characteristic class $c$ on the vertical tangent bundle $T_v(p)$ of the map $p$. In particular, the tautological classes are defined on the universal smooth oriented fibre bundle with fibre $M$, whose base is the classifying space $\B\mathrm{Diff}^+(M)$ for the topological group of orientation-preserving diffeomorphisms of $M$, yielding universal classes
\[\kappa_c(M) \in H^{k-d}(\B\mathrm{Diff}^+(M);R).\]
These classes were first considered in the case that $M$ is an oriented surface, where they have been studied in detail by both algebraic geometers and topologists \cite{Mumford, Morita, Miller, Looijenga, Faber}. They were the subject of Mumford's conjecture describing the rational cohomology of the stable moduli space of Riemann surfaces, which was resolved in the work of Madsen, Tillmann, and Weiss \cite{MW, MT}.
In higher dimensions, tautological classes have been of recent interest due to the work of Galatius and Randal-Williams, culminating in \cite{GRW3}, which describes the rational cohomology of $\B\DIFF^+(M)$ in terms of tautological classes for certain simply connected manifolds $M$ of dimension $2n \geq 6$, in a range bounded by roughly half the genus of $M$; the genus of $M$ refers to the number of $S^n\times S^n$ connect-summands of $M$. In fact, 
already their work in \cite{GRW2} and \cite{GRW1} implies that any oriented $2n$-manifold of genus at least 11 has non-trivial tautological classes!
\newline

The goal of the present paper is to discuss tautological classes for aspherical manifolds. Aspherical manifolds of dimension $2n > 2$ have vanishing genus in the sense just described (a $S^n \times S^n$ connect-summand gives elements of $n$th homotopy whose nontriviality may be detected using the intersection form) so the results mentioned above reveal nothing in this case. 

Our main theorem will be stated in terms of the following two conjectures.
\begin{conjbbc*}
For a closed aspherical manifold $M$ the canonical map $\TOPtw(M) \to \mathcal G(M)$ is a weak homotopy equivalence.
\end{conjbbc*}

Here $\TOPtw(M)$ denotes the realisation of the semi-simplicial set of \emph{block homeomorphisms} of $M$, and $\mathcal G(M)$ denotes the space of self homotopy equivalences of $M$. The block Borel conjecture is a strong form of the uniqueness part of the classical Borel conjecture: that conjecture says that a homotopy equivalence between aspherical manifolds is homotopic to a homeomorphism; the block Borel conjecture says that the space of homotopy equivalences between homeomorphic aspherical manifolds is equivalent to the space of block homeomorphisms. 
For the purposes of this introduction the most important feature of this conjecture is that for manifolds of dimension at least $5$ it is implied by the Farrell--Jones conjectures, and thus is known for large swathes of aspherical manifolds by the work of Bartels, Reich, L{\"u}ck, and many others \cite{BLR, BL, KLR}. 

Another input into our work is Burghelea's conjecture \cite{Burghelea}, the part of which relevant for us reads as follows. 

\begin{conjbc*}
For a closed aspherical manifold $M$ and a central element $g \in \pi_1(M)$ the rational cohomological dimension with trivial coefficients of 
$\pi_1(M) / \langle g \rangle$ is finite.
\end{conjbc*}
This conjecture is not as well studied as the Farrell--Jones conjecture, but is still known to hold for a large class of groups. Finally, let $\DIFF_h(M) \leq \DIFF(M)$ denote the subgroup of those diffeomorphisms \emph{homotopic} to the identity; recall that a smooth fibre bundle with fibre $M$ has structure group $\DIFF_h(M)$ if and only if its fibre transport along any loop is homotopic to the identity. With these preliminaries out of the way we can state our main result.

\begin{maintheorem}
If an oriented, smooth, closed, aspherical manifold $M$ of dimension $d$ satisfies the central part of Burghelea's conjecture and the block Borel conjecture, then for all smooth $M$-fibre bundles $p\colon E \to B$ with structure group $\DIFF_h(M)$
, we have
\[ 0 = \kappa_c(p) \in H^{k-d}(B;\Q) \]
for all $c \in H^k(\BSO(d);\Q)$ with $k \neq d$. 
\end{maintheorem}

In particular, as any smooth $M$-fibre bundle over a simply connected base space admits a reduction of its structure group to the identity component of $\DIFF(M)$ 
the tautological classes of such bundles vanish.

The hypothesis on the structure group cannot be completely relaxed. When $M$ is an orientable surface of large genus it is well known (see e.g.\ \cite{Miller}) that many tautological classes are non-zero, so by taking products we obtain examples of aspherical manifolds of any even dimension having non-zero tautological classes.



This theorem in particular recovers several recent vanishing theorems of Bustamante, Farrell, and Jiang \cite{BFJ}, but applies to a much wider class of manifolds. At the end of the paper we shall describe conditions on the fundamental group of an aspherical manifold which are known to imply that $M$ satisfies both relevant conjectures. 

This theorem is not the strongest or most general result that we prove, but is the most easily stated and has the least technical hypotheses. We shall prove similar vanishing results under conditions weaker than the block Borel conjecture, these will also hold for topological block bundles, in certain situations will extend to cover the case $k=d$ or diffeomorphisms not homotopic to the identity, and we also have results for more general coefficients. To give some idea of these statements it will be helpful to first go through the main ingredients of the proof, but the strongest formulations will only be given in the body of the text.

\subsection{Characteristic classes for topological block bundles}

The first step in our proof is to show that tautological classes can be defined 
not just for smooth fibre bundles but for topological block bundles. This extends earlier work of Ebert and Randal-Williams \cite{ERW, RW}, where among other things they show that rational tautological classes can be defined both for topological fibre bundles and for smooth block bundles. 

To this end we will consider the universal oriented $M$-block bundle $\pi \colon \Ep(M) \to \B\TOPtw{\mathstrut}^{\>\!+}(M)$ and construct an oriented stable vertical tangent bundle $T_v^s (\pi) \colon \Ep(M) \to \BSTOP$. 
We also construct a fibrewise Euler class $\e(\pi) \in H^d(\Ep(M);\Z)$. In fact, we construct this class for any oriented fibration whose fibre is a Poincar\'e duality space of formal dimension $d$. 
By pulling cohomology classes back along the map 
\[(T^s_v(\pi), \e(\pi)) \colon \Ep(M) \lto \BSTOP \times K(\Z,d)\] 
and applying the Gysin homomorphism, we can associate 
\[\kappa_c(M) = \pi_!((T^s_v(\pi), \e(\pi))^*(c)) \in H^{k-d}(\B\TOPtw{\mathstrut}^{\>\!+}(M);R)\]
to a cohomology class $c \in H^k(\BSTOP \times K(\Z,d);R)$. 

These define characteristic classes of oriented block bundles, and together with the stable vertical tangent bundle and fibrewise Euler class can be pulled back from the universal oriented $M$-block bundle to any other. 
On a block bundle which arises from a smooth fibre bundle $p \colon E \to B$, $T_v^s(p)$ is the stabilisation of the vertical tangent bundle, $\e(p)$ is the Euler class of the vertical tangent bundle, and the Gysin homomorphism is the usual one, so these tautological classes reduce to those of the same name defined earlier. We will show that they also agree with the constructions of \cite{ERW} and \cite{RW}. This comparison, in particular, shows that the classes defined in \cite{ERW} lie in the image of the Gysin homomorphism, a point not addressed in \cite{ERW} but essential for our work. 

Recall now that $H^*(\BSO(d);\mathbb{Q})$ is generated by Pontryagin and Euler classes, and by work of Novikov, Kirby and Siebenmann the rational Pontryagin classes are pulled back from $\BSTOP$. Therefore, to establish a vanishing result for rational tautological classes it suffices to consider topological block bundles. That is, writing $\TOPtw_h(M) \leq \TOPtw(M)$ for those components represented by homeomorphisms homotopic to the identity, it is enough to show that 
\[0 = \kappa_c(M) \in H^{k-d}(\BTOPtw_h(M);\mathbb{Q})\] 
for all $c \in H^k(\BSTOP \times K(\mathbb{Z},d);\mathbb{Q})$ such that $k \neq d$. Assuming the two conjectures stated earlier, we will show the vanishing of these classes, as we now explain.

\subsection{Vanishing results}

To this end let us fix a closed, connected, oriented, aspherical topological manifold $M$ which satisfies the block Borel conjecture, i.e.\ such that the map
\[\TOPtw(M) \longrightarrow \mathcal{G}(M)\]
is a weak equivalence. This means that topological $M$-block bundles which are fibre homotopy equivalent are in fact equivalent as block bundles, and in particular means that $\TOPtw_h(M)$ agrees with the component of the identity $\TOPtw_0(M)$. As discussed in the last section the stable vertical tangent bundle of a manifold bundle only depends on the underlying topological block bundle, so it is fibre homotopy invariant among $M$-bundles. This conclusion was obtained in \cite{BFJ} by a 
different route. Together with our construction of the fibrewise Euler class, it 
implies that rational tautological classes for $M$-fibre bundles are invariant under fibre homotopy equivalences and therefore vanish on fibre homotopically trivial bundles.

To obtain a criterion for fibre homotopy triviality note that for any connected, aspherical complex $X$ a straightforward computation shows 
\[\pi_k\mathcal(\G(X)) = \left\{\begin{matrix}\Out(\pi_1(X)) & k = 0 \\ \C(\pi_1(X)) & k = 1 \\ 0 & k \geq 2\end{matrix}\right.\] 
where $\Out$ denotes the outer automorphism group and $\C$ the centre of a given group. As our results only concern the homotopy type (resp.\ the homology) of the classifying space of $\TOPtw_h(M)$ it is in fact enough that a weaker property than the block Borel conjecture should hold: that the map
\[\B\TOPtw_h(M) \longrightarrow \B\G_0(M)\]
be a weak equivalence (resp.\ induce an isomorphism on $R$-homology, for some ring of coefficients $R$). We dub this the \emph{{\pbc} (resp.\ with $R$-coefficients)}.
In distinction with the block Borel conjecture, it is implied by the Farrell--Jones conjectures also when the aspherical manifold in question is of dimension $4$. Now if $\C(\pi_1(M)) = 0$, then $\B\G_0(M)$ is contractible; we refer to such manifolds as \emph{centreless} and a block bundle with centreless, aspherical fibre is thus fibre homotopically trivial. We therefore find: 

\begin{Thm*}
If $M$ is a closed, oriented, aspherical, centreless manifold which satisfies the {\pbc} with $R$-coefficients, then 
\[0 = \kappa_c(M) \in H^{k-d}(\BTOPtw_h(M);R)\] 
for all $c \in H^k(\BSTOP \times K(\mathbb{Z},d);R)$ such that $k \neq d$.
\end{Thm*}
The consequences of this theorem for smooth manifold bundles, while not explicitly stated there, were essentially already obtained in \cite{BFJ}. And while the methods are similar as well, our approach offers a novel perspective: The tautological classes of bundles with centreless, aspherical fibre and fibre transport homotopic to the identity vanish because the universal space in which they are defined is contractible by the block Borel conjecture. As explained above, the result in particular implies the vanishing of all rational tautological classes (in positive degree) for a smooth fibre bundle satisfying the assumptions. The implications for the \emph{integral} tautological classes of smooth fibre bundles are somewhat delicate, as $H^*(\BSTOP\times K(\mathbb{Z},d);\Z) \to H^*(\BSO(d);\Z)$ is not surjective. Instead of their vanishing, one only obtains (somewhat inexplicit) universal bounds on their order.

The condition that $c$ not have degree $d$ cannot be removed, as already observed in \cite{BFJ}: Because every bordism class can be represented by a negatively curved manifold, see \cite{Ontaneda}, for $c \in H^d(\BSO(d);\mathbb Q)$ the classes $\kappa_c(M) = \langle c(TM),[M]\rangle$ do not generally vanish on aspherical manifolds. However, any negatively curved manifold is centreless: We will now see that stronger results may be obtained for an aspherical manifold that satisfies the {\pbc} with $\Q$-coefficients, whose fundamental group has non-trivial centre, and in addition satisfies the central part of Burghelea's conjecture.

We begin by observing that by the {\pbc} with $R$-coefficients the underlying fibration of the universal $M$-block bundle with fibre transport homotopic to the identity is $R$-homology equivalent to
\[\pi\colon \B(\Gamma / \C(\Gamma)) \lra \B^2 \C(\Gamma)\]
where we have abbreviated
$\Gamma := \pi_1(M)$ and the map $\pi$ classifies the central extension
\[1 \longrightarrow \C(\Gamma) \longrightarrow \Gamma \longrightarrow \Gamma / \C(\Gamma) \longrightarrow 1.\]
This observation relates the Gysin map for the universal block bundle over $\BTOPtw_h(M)$ with the central part of Burghelea's conjecture, which we shall use to show the following.

\begin{Thm*}
If $\Gamma$ is a rational Poincar\'e duality group of dimension $d$ with non-trivial centre, which satisfies the central part of Burghelea's conjecture, then the Gysin map
\[\pi_!\colon H^*(\B(\Gamma / \C(\Gamma)); \mathbb Q) \longrightarrow H^{*-d}(\B^2 \C(\Gamma); \mathbb Q)\]
vanishes. If $C(\Gamma)$ is finitely generated, then the same statement holds integrally.
\end{Thm*}

It seems to be an open problem whether the centre of the fundamental group of an aspherical manifold is finitely generated, though this is known for several classes of groups.

\begin{corollary*}
Let $M$ be a closed, connected, oriented, aspherical manifold with non-trivial centre that satisfies the {\pbc} with $\Q$-coefficients and the central part of Burghelea's conjecture. Then 
\[0 = \kappa_c(M) \in H^{k-d}(\BTOPtw_h(M);\Q)\] 
for all $c \in H^k(\BSTOP \times K(\mathbb{Z},d);\Q)$. If $C(\pi_1(M))$ is finitely generated then the same statement holds integrally.
\end{corollary*}

This result immediately implies the vanishing of all tautological classes of all smooth fibre bundles with fibres satisfying the hypotheses. Even if one is only interested in smooth fibre bundles, it seems essential to consider block bundles in order to prove it. This result concerns \emph{all} tautological classes, not just those of non-zero degree, which means that it has content even for the bundle $M \to *$. 
\begin{corollary*}
Let $M$ be as in the previous corollary. Then the Euler characteristic and all Pontryagin numbers of $M$ vanish.
\end{corollary*}

The vanishing of the Euler characteristic in the situation of the corollary was obtained by Gottlieb in \cite{Gottlieb} by more elementary means, without assuming either conjecture. We believe that the vanishing of Pontryagin numbers is new; it means that $M$ represents a torsion element in the topological oriented cobordism ring and an element of order at most $2$ in the smooth one when smooth itself. This should be contrasted with Ontaneda's result mentioned above. Let us also mention that the result is trivial if an element of the centre of $\pi_1(M)$ can be realised by a principal $S^1$-action (e.g. $M$ a nilmanifold), but that this need not happen in general \cite{CWY}.

The principal examples for which we verify the hypotheses of the two corollaries are manifolds built as iterated bundles with fibres either non-positively curved manifolds or biquotients of Lie groups (that is manifolds of the form $\Gamma \backslash G / K$, where $\Gamma$ is a cocompact lattice and $K$ is a maximal compact subgroup). During the proof of the above theorem we will unearth slightly weaker finiteness conditions than Burghelea's that still allow the proof of vanishing of the Gysin map to go through. Chief among the examples we can cover this way is $S^1 \times M$, whenever $\pi_1(M)$ is a Farrell--Jones group.
\newline

This discussion leads us to formulate the following
\begin{conj*}
Let $M$ be a closed, connected, oriented, aspherical manifold. If $\C(\pi_1(M)) \neq 0$ then, for any ring $R$,
\[0 = \kappa_c(M) \in H^*(\B\TOPtw_h(M);R)\]
for all $c \in H^*(\BSTOP\times K(\Z,d);R)$.
\end{conj*}
Since $H^*(\BSTOP; \mathbb Z/2) \rightarrow H^*(\BSO; \mathbb Z/2)$ is surjective, this conjecture in particular implies that all Stiefel--Whitney numbers vanish, and thus that a smooth aspherical manifold with non-trivial centre is nullbordant.

\subsection*{Organisation of the paper}

We begin \cref{sec2} by recalling basics about block bundles and then construct the universal stable vertical tangent bundle in the latter half, the fibrewise Euler class for a fibration with Poincar\'e fibre in \cref{sec:Euler}, and tautological classes for block bundles in \cref{sec3}. We also compare our definitions to previous ones.
In \cref{sec4} we review the homotopy type of the space of block homeomorphisms and its relation to the Farrell--Jones conjectures. Along the way we obtain the main theorem in the centreless and the abelian case. To discuss general aspherical manifolds whose centre is non-trivial, we introduce a plethora of finiteness conditions in \cref{sec5}, among them Burghelea's conjecture, and untangle their relations, in particular proving our main vanishing results. Finally, in \cref{sec6} we discuss several classes of manifolds which satisfy both conjectures and indeed prove the vanishing of tautological classes for a few cases not covered by the existing literature on the Burghelea conjecture via intermediate finiteness assumptions introduced in \cref{sec5}. We end 
some open questions which we encountered on the way.

\begin{ackn}
We are happy to thank Diarmuid Crowley, S\o{}ren Galatius, and Wolfgang Steimle for several helpful discussions along the way. Furthermore, we want to heartily thank Alexander Engel and Micha{\l} Marcinkowski for making us aware of Burghelea's conjecture in the first place, which replaces a stronger assumption in an earlier version of this paper.
We also want to thank the anonymous referee for pointing out an oversight in a previous version of this paper concerning the relation between the various finiteness properties in \cref{sec5}. Finally, we would like to express our gratitude towards the Mathematisches Forschungsinstitut Oberwolfach, where the seeds of the present project were sown over a lovely barbecue. 

FH and ML enjoyed support of the CRC 1085 `Higher invariants' at the University of Regensburg, FH and WL are members of the Hausdorff Centre for Mathematics, DFG GZ 2047/1, project ID 390685813 at the University of Bonn and WL and ML were supported by the ERC-grant 662400 `KL2MG-interactions'. ORW was supported by EPSRC grant EP/M027783/1 `Stable and unstable cohomology of moduli spaces'.  
\end{ackn}

\section{A stable vertical tangent bundle for block bundles}\label{sec2}

In this section we shall remind the reader of the definition of a block bundle with fibre a manifold $M$, describe the classifying space for such block bundles and the universal block bundle, and construct the stable vertical normal bundle on its total space. For our applications we require this theory for topological manifolds and topological block bundles, but it can be developed in any category $\CAT \in \{\DIFF, \PL, \TOP\}$ and we shall do so in this generality.

Many of the necessary ideas already appeared in work of Ebert and Randal-Williams \cite{ERW}, where models for the universal smooth block bundle were described, and it was shown that any smooth block bundle over a finite simplicial complex had a stable vertical tangent bundle. The argument given there was particular to vector bundles (gluing together explicit maps to Grassmannians defined on different blocks). Here we shall improve the result to hold for $\CAT$ block bundles and give a stable vertical $\CAT$ tangent bundle for the universal block bundle (whose base is \emph{not} a finite simplicial complex).

The credulous reader not interested in the rather technical construction of the universal vertical tangent bundle may skip the entire section, except maybe the reminder on block bundles in \cref{bla} if warranted, since the techniques employed are entirely different from those of the remainder of the article. In particular, they will not miss out on anything else relevant.

\subsection{Notation and conventions}

For convenience we use the following notion. A \emph{$p$-block space} is a space $X$ with a reference map $\pi \colon X \to \Delta^p$ to the $p$-simplex. A morphism between $p$-block spaces $(X, \pi)$ and $(X', \pi')$ is a continuous map $f \colon X \to X'$ which \emph{weakly} commutes with the reference map in the following sense: for each face $\tau \subset \Delta^p$, the map $f$ sends $\pi^{-1}(\tau)$ into $\pi'^{-1}(\tau)$. If $X$ and $X'$ are $\CAT$ manifolds and $f$ is a $\CAT$ isomorphism, we say it is a \emph{$p$-block $\CAT$ isomorphism}.

If $(X, \pi)$ is a $p$-block space then for each $i=0,1,2,\ldots,p$ we obtain a $(p-1)$-block space $d_i(X, \pi)$ by restriction to the $i$th face of $\Delta_p$. More precisely, if $\Delta^{p-1}_i \subset \Delta^p$ denotes the face spanned by all vertices but the $i$th, then $d_i(X, \pi) = (\pi^{-1}(\Delta^{p-1}_i), \pi\vert_{\pi^{-1}(\Delta^{p-1}_i)})$. We call this the \emph{restriction} of $X$ to the $i$th face of $\Delta^p$.

We shall always implicitly consider spaces of the form $\Delta^p \times T$ to be $p$-block spaces with reference map given by projection to the first factor.

\subsection{Block diffeomorphisms}

For $i=0,1,\ldots,p$
and $0 < \epsilon \leq 1$ let us write
\[\Delta^p_i(\epsilon) := \{(t_0, t_1, \ldots, t_p) \in \Delta^p \,\vert\, 0 \leq t_i < \epsilon\}.\]
For any $0 < \epsilon \leq 1$ define a homeomorphism
\begin{align*}
h_i(\epsilon) \colon \Delta^p_i(\epsilon)   &\lra \Delta^{p-1}_i \times [0,\epsilon) \\
(t_0, t_1, \ldots, t_p) &\longmapsto (\tfrac{t_0}{1-t_i}, \tfrac{t_1}{1-t_i}, \ldots, \tfrac{t_{i-1}}{1-t_i}, \tfrac{t_{i+1}}{1-t_i}, \ldots, \tfrac{t_p}{1-t_i}; t_i ).
\end{align*}
and a retraction $\pi_i(\epsilon) = \pi_1 \circ h_i(\epsilon) \colon \Delta^p_i(\epsilon) \to \Delta^{p-1}_i$.

\begin{Def}\label{def:pBlockCatIso}
A \emph{collared $p$-block $\CAT$ isomorphism of $\Delta^p \times M$} is a $\CAT$ isomorphism
\[f \colon \Delta^p \times M \lra \Delta^p \times M\]
which is also a $p$-block map, such that for each $i=0,1,\ldots, p$ there is an $\epsilon>0$ such that $f$ preserves the set $\Delta^p_i(\epsilon) \times M$ and $h_i(\epsilon) \circ f\vert_{\Delta^p_i(\epsilon) \times M} \circ h_i(\epsilon)^{-1} = d_i(f) \times \mathrm{Id}_{[0,\epsilon)}$.
\end{Def}
It is an elementary but tedious exercise to see that if $f$ is a collared $p$-block $\CAT$ isomorphism of $\Delta^p \times M$ then $d_i(f)$ is a collared $(p-1)$-block $\CAT$ isomorphism of $\Delta^{p-1} \times M$. Thus there is a semi-simplicial group $\widetilde{\CAT}(M)_\bullet$ with $p$-simplices the set of collared $p$-block $\CAT$ isomorphisms of $\Delta^p \times M$, and face maps given by restriction. 
The classifying space $\B \widetilde{\CAT}(M)$ is defined to be the geometric realisation of the bi-semi-simplicial set $N_\bullet \widetilde{\CAT}(M)_\bullet$ obtained by taking the levelwise nerve of the semi-simplicial group $\widetilde{\CAT}(M)_\bullet$. 

\begin{Rmk}\label{Remark Kan}
It is claimed in \cite[Appendix A, Section 3]{BLR3} that $\widetilde{\CAT}(M)_\bullet$ can be enhanced to a simplicial group (which would, in particular, imply the Kan property). The construction of degeneracy maps given, however, is not compatible with the collaring conditions. In case $\CAT \in \{\PL,\TOP\}$ one can simply drop the collaring condition to fix this issue, as done in \cite{ERW} (this change clearly does not affect the homotopy type of $\widetilde{\CAT}(M)_\bullet$). However, the proposed degeneracy maps also fail to be smooth. In fact, contrary to a claim in the proof of \cite[Proposition 2.8]{ERW}, $\widetilde{\DIFF}(M)_\bullet$ fails to be Kan without the collaring condition: Not even horns in the $2$-simplex need to be fillable, unless some compatibility on derivatives is enforced at the intersection of the given faces. We thank Manuel Krannich for making us aware of these oversights. 

We now argue that $\widetilde{\CAT}(M)_\bullet$ is a Kan semi-simplicial set. A map $\Lambda^n_i \to \widetilde{\CAT}(M)_\bullet$ corresponds to a $\CAT$-isomorphism $\phi \colon \Lambda_i^n \times M \to \Lambda_i^n \times M$ with the collaring condition. We extend $\phi$ to a collared $\CAT$-isomorphism
\[\phi_\epsilon \colon \Lambda_i^n(\epsilon) \times M \longrightarrow  \Lambda_i^n(\epsilon) \times M\]
where $\Lambda_i^n(\epsilon)  = \bigcup_{i\neq j} \Delta_j^n(\epsilon)$ with $\epsilon$ the minimum of the collars respected by the restrictions of $\phi$ to the faces of $\Lambda_i^n$. Then one can pick a suitable embedding $g \colon \Delta^n \rightarrow \Lambda_i^n(\epsilon)$ and conjugate $\phi_\epsilon$ with $g \times \mathrm{id}_M$ to obtain an extension of $\phi$ as desired.

Sufficient conditions for $g$ being suitable are, for example, given as follows: Let $p \colon \Delta^n \rightarrow \Delta^n_i$ denote the linear map which sends the $i$th vertex to the barycentre of $\Delta^n_i$ and all other vertices to themselves. Furthermore, for a face $T$ of a simplex $S$ denote by $\Lambda_T^S$ the union of all faces of $S$ not containing $T$, and let $\Lambda_T^S(\epsilon)$ denote an $\epsilon$ neighbourhood of $\Lambda_T^S$ just as above. Now for $T$ a simplex of $\Lambda_i^n$ put 
\[B_T = \bigcup_{S \supseteq T} p\big(S \setminus \Lambda_T^S(\epsilon)\big) \subseteq \Delta_i^n,\]
where $S$ runs over the $(n-1)$-dimensional faces of $\Lambda_i^n$ containing $T$. Then there should exist a $\delta>0$ such that
\begin{enumerate}
\item $g$ restricts to the identity on $\Lambda_i^n(\delta)$,
\item for every top-dimensional face $T$ of some $\Delta_j^n$ with $\epsilon$ thickening $T(\epsilon) \subseteq \Delta_j^n$, we have
\[gh_i^{-1}(pT(\epsilon) \times [0,\delta)) \subseteq \bigcap_{T \subset \Delta_k^n} \Delta_k^n(\epsilon),\]
where $h_i$ is the homeomorphism described in the paragraph before \cref{def:pBlockCatIso},
\item for each (not necessarily top-dimensional) face $T$ of $\Lambda_i^n$ the composite
\[g_j^T \colon (p(\Delta_j^n) \cap B_T) \times [0,\delta) \xrightarrow{h_i^{-1}} \Delta_i^n(\delta) \xrightarrow{g} \Delta_j^n(\epsilon)\]
has the property that 
\[\mathrm{pr}_1 h_j g^T_j(x+Dp(v),t) = \mathrm{pr}_1 h_jg^T_j(x,t)+v \in \Delta_j^n\]
for every tangent vector $v$ of $T$ such that $x+Dp(v) \in p(\Delta_j^n) \cap B_T$, and 
\item finally
\[\mathrm{pr}_1 h_T g^T_j(x,t) = \mathrm{pr}_1 h_T g^T_j(x,0),\]
where \[h_T \colon \bigcap_{T \subset \Delta_k^n} \Delta_k^n(\epsilon) \longrightarrow T \times [0,\epsilon)^{n-\mathrm{dim}(T)}\]
is given by iterating the $h_l$.
\end{enumerate}
Such $g$ is readily constructed for $n=1,2$, in a way isotopic to the identity. To obtain it for higher $n$ note that $g$ satisfying conditions (3) and (4) can be chosen linear on the various $h_i^{-1}(p(\Delta_j^n \setminus \partial\Delta_j^n(\epsilon)) \times [0,\delta))$, where $0 \leq j \leq n$, $i \neq j$. To extend such a $g$ to the remainder of $\Delta^n$, note that in $\epsilon$-thickenings of $h_i^{-1}(p(T) \times [0,\delta))$ for lower, but positive dimensional, faces $T$ of $\Lambda_i^n$ the extension problem reduces to the construction of $g$ for some lower value of $n$, again by (3). Finally, the arising embedding can be extended to all of $h_i^{-1}(\Delta_i^n \times [0,\delta))$, and then all of $\Delta^n$ by iterated application of the isotopy extension theorem, as no further invariance conditions need to be met.
\end{Rmk}

\subsection{Block bundles and their moduli spaces}\label{bla}

Let $K$ be a simplicial complex, and $\pi \colon E \to \vert K \vert$ be a continuous map. We recall 
the notion of a $\CAT$ block bundle structure on this map, with fibre a $\CAT$ manifold $M$. A \emph{block chart} for $E$ over a simplex $\sigma \subset \vert K \vert$ is a homeomorphism
\[h_\sigma \colon \pi^{-1}(\sigma) \lra \sigma \times M\]
such that for each face $\tau \leq \sigma$ the map $h_\sigma\vert_{\pi^{-1}(\tau)}$ sends $\pi^{-1}(\tau)$ homeomorphically to $\tau \times M$. A \emph{block atlas} $\mathcal{A}$ for $E$ is a set of block charts for $E$, at least one for each simplex of $\vert K \vert$, so that if $h_{\sigma_i} \colon \pi^{-1}(\sigma_i) \to \sigma_i \times M$, $i=0,1$, are two block charts then the composition
\[h_{\sigma_1} \circ h_{\sigma_0}^{-1} \colon (\sigma_0 \cap \sigma_1) \times M \lra (\sigma_0 \cap \sigma_1) \times M\]
is a $p$-block $\CAT$ isomorphism in the sense of \cref{def:pBlockCatIso}. A \emph{block bundle structure} on $\pi \colon E \to \vert K \vert$ is a maximal block atlas. 

It can be shown directly that concordance classes of block bundles over $\vert K \vert$ are classified by homotopy classes of maps $f\colon \vert K \vert \to \B\widetilde{\CAT}(M)$, but for both the proof and geometric constructions, the following model for the classifying space is more convenient. It depends on $\CAT \in \{\DIFF, \TOP, \PL\}$, but we omit this from the notation.

\begin{Def}\label{def:MM}
Let $\mathcal{M}(M)^{\epsilon, n}_p$ denote the set of locally flat $\CAT$ submanifolds $W \subset \Delta^p \times \bR^n$ (considered as $p$-block spaces via projection to the $\Delta^p$ factor) such that for each $i=0,1,\ldots, p$ we have
\begin{enumerate}[(i)]
\item $W$ is $\CAT$ transverse to $\Delta^{p-1}_i \times \bR^n \subset \Delta^{p} \times \bR^n$, 

\item\label{it:Collar} $W \cap (\Delta^p_i(\epsilon) \times \bR^n) = (\pi_i(\epsilon) \times \mathrm{Id}_{\bR^n})^{-1}(W\cap(\Delta_i^{p-1} \times \bR^n))$, and
\item there is a $p$-block $\CAT$ isomorphism $f \colon \Delta^p \times M \to W \subset \Delta^p \times \bR^n$ which is collared in the sense that for each $i=0,1,\ldots,p$ the map $f$ agrees with the map
\begin{align*}
(\Delta^{p-1}_i * \{e_i\}) \times M &\lra (\Delta^{p-1}_i * \{e_i\}) \times \bR^n\\
((1-t_i)\cdot w + t_i \cdot e_i , x) &\longmapsto ((1-t_i)\cdot w' + t_i \cdot e_i, x')
\end{align*}
on $\Delta^{p}_i(\epsilon) \times M$, where $(w',x') = f\vert_{\Delta^{p-1}_i \times M}(w,x)$ and $e_i \in \mathbb R^n$ denotes the $i$-th unit vector.
\end{enumerate}
Define face maps $d_i \colon \mathcal{M}(M)^{\epsilon,n}_p \to \mathcal{M}(M)^{\epsilon,n}_{p-1}$ by restricting $W$ to the $i$th face of $\Delta^p$, to give a semi-simplicial set $\mathcal{M}(M)^{n}_\bullet$. Put $\mathcal M(M)^n_\bullet = \bigcup_{\epsilon >0} \mathcal M(M)^{\epsilon,n}_\bullet$ and finally let $\mathcal{M}(M)_\bullet = \colim\limits_{n \to \infty} \mathcal{M}(M)^{n}_\bullet$, under the evident comparison maps, and $\mathcal{M}(M) = \vert \mathcal{M}(M)_\bullet\vert$.
\end{Def}

The semi-simplicial set $\mathcal{M}(M)_\bullet$ is Kan: given a  $E \subset \Lambda^p_i \times \bR^n$ defining a block bundle over a horn $\Lambda^p_i$ to be extended to $\Delta^p$, condition (\ref{it:Collar}) above gives an extension to an open neighbourhood of $\Lambda^p_i$, and a full extension may be obtained from this by choosing an isotopy from the identity map of $\Delta^p$ to a suitable embedding into this open neighbourhood, as in \cref{Remark Kan}.

To compare $\mathcal{M}(M)$ with $\B\widetilde{\CAT}(M)$, we follow \cite[Proposition 2.3]{ERW} and consider the bi-semi-simplicial set $X_{\bullet,\bullet}$ with $(p,q)$-simplices given by a $W \in \mathcal{M}(M)_q$ and a sequence
\[W \overset{f_0}\longleftarrow \Delta^q \times M  \overset{f_1}\longleftarrow \Delta^q \times M  \overset{f_2}\longleftarrow \cdots  \overset{f_p}\longleftarrow \Delta^q \times M\]
of $q$-block $\CAT$ isomorphisms, where $f_1, \ldots, f_p$ are collared as in \cref{def:pBlockCatIso}, and $f_0$  is collared as in \cref{def:MM}.
The face maps in the $q$ direction are by restriction to faces, and those on the $p$ direction are by composing the $f_i$ or forgetting $f_p$. The augmentation map $X_{\bullet,q} \to \mathcal{M}(M)_q$, which just records $W$, has fibre over $W$ isomorphic to $E_\bullet G$, where $G$ is the group of the collared $q$-block $\CAT$ isomorphisms of $\Delta^q \times M$; thus $|X_{\bullet,q}| \overset{\simeq}\to \mathcal{M}(M)_q$. There is a map $X_{p, \bullet} \to N_p \widetilde{\CAT}(M)_\bullet$,  which just records $(f_1, \ldots, f_p)$. This is a Kan fibration of semi-simplicial sets, and as in the proof of \cite[Proposition 2.3]{ERW} its fibre after geometric realisation can be described as the space of block embeddings of $M$ into $\bR^\infty$, which is contractible. In total this yields a preferred homotopy equivalence $\mathcal{M}(M) \simeq \B\widetilde{\CAT}(M)$.

Let us now describe the universal $M$-block bundle $\pi \colon \mathcal E(M) \to \mathcal M(M)$. Strictly speaking this will not be a block bundle as described in the beginning of this section, since $\mathcal {M}(M)$ is not a finite simplicial complex. We will, however, blur this distinction in the notation, as the pull back of $\pi$ along a simplicial map from a finite simplicial complex is indeed a block bundle as in the proof of \cite[Proposition 2.7]{ERW}.  

Let $\mathcal{E}(M)_p \subset \mathcal{M}(M)_p \times \Delta^p \times \bR^\infty$ be the subspace of those triples $(W; t_0, \ldots, t_p ; x)$ for which $(t_0, \ldots, t_p ; x) \in W$, and let $\pi_p \colon \mathcal{E}(M)_p \to \mathcal{M}(M)_p \times \Delta^p$ denote the projection map. These assemble to a continuous map
\[\pi \colon \mathcal{E}(M) \lra \mathcal{M}(M)\]
where 
\[\mathcal E(M) = \left(\bigsqcup_{p \geq 0} \mathcal{E}(M)_p \right)/\sim\]
with $\sim$ the equivalence related generated by
\[(W; t_0, \ldots, t_{i-1}, 0, t_{i+1}, \ldots t_p ; x) \sim (d_i(W); t_0, \ldots, t_{i-1}, t_{i+1}, \ldots t_p ; x)\]
and
\[\mathcal M(M) = \left(\bigsqcup_{p \geq 0} \Delta^p \times \mathcal{M}(M)_p \right)/\sim\]
the usual geometric realisation. The preimage of the simplex $\{W\} \times \Delta^p \subset  \mathcal{M}(M)$ is $\{W\} \times W$, which is $p$-block $\CAT$ isomorphic to $\Delta^p \times M$.

We will now show that the map $\pi \colon \mathcal E(M) \to \mathcal M(M)$ is a weak quasi-fibration, in the sense that the comparison map $\pi^{-1}(v) \to \mathrm{hofib}_v(\pi)$ is a weak homotopy equivalence for any vertex $v \in \mathcal{M}(M)_0$, thereby directly identifying the underlying fibration of the universal block bundle. For future use, we formulate this in a slightly more general manner.

\begin{Prop}\label{prop:quasifib}
If $X_\bullet$ is a semi-simplicial set and $f \colon X_\bullet \to \mathcal{M}(M)_\bullet$ a semi-simplicial map, then the map $f^*\pi \colon f^*\mathcal E(M) \to |X_\bullet|$ is a weak quasi-fibration.
\end{Prop}
\begin{proof}
Let us first suppose that $X_\bullet$ is a finite semi-simplicial set. We proceed by double induction on the dimension of $X_\bullet$ and the number of top-dimensional simplices. Firstly, if $|X_\bullet|$ is 0-dimensional then the claim clearly holds. Otherwise, let $\sigma \in X_p$ be a top-dimensional simplex and $X'_\bullet$ be the semi-simplicial set obtained by removing $\sigma$, and write $f' = f\vert_{X'_\bullet}$. Then $f(\sigma) \in \mathcal{M}(M)_p$ is a submanifold of $\Delta^p \times \bR^\infty$ which is $p$-block isomorphic to $\Delta^p \times M$. Let us write $\partial f(\sigma) = f(\sigma) \cap (\partial \Delta^p \times \bR^\infty)$. There is a cube
\begin{equation*}
\xymatrix@C=.3cm@R=.3cm{
 & \partial f(\sigma)  \ar'[d]^-b[dd]\ar[ld] \ar[rr] && f(\sigma) \ar[ld] \ar[dd]^-c\\
 (f')^*\mathcal E(M)  \ar[rr] \ar[dd]^-a & & f^*\mathcal E(M) \ar[dd] \\
 & \partial \Delta^p  \ar[ld] \ar'[r][rr] && \Delta^p \ar[ld]\\
 |X'_\bullet| \ar[rr] & & |X_\bullet|
}
\end{equation*}
in which the top and bottom faces are homotopy push-outs. As $f(\sigma)$ is $p$-block isomorphic to $M \times \Delta^p$, the map $c$ is a weak quasi-fibration; as $X'_\bullet$ has fewer top-dimensional simplices than $X_\bullet$ we may suppose by induction that $a$ is a weak quasi-fibration; as $\partial \Delta^p$ is of lower dimension than $X_\bullet$ we may suppose by induction that $b$ is a weak quasi-fibration. The left and back faces are cartesian, so as $a$, $b$, and $c$ are weak quasi-fibrations it follows that they are homotopy cartesian. By Mather's First Cube Theorem \cite[Theorem 18]{Mather} it follows that the front and right faces are also homotopy cartesian: as $c$ (or $a$) is a weak quasi-fibration, it follows that $f^*\pi$ is too.

Now, if $X_\bullet$ is an arbitrary semi-simplicial set, let $v \in X_0$ and let $\mathcal{F}$ denote the directed set of finite sub-semi-simplicial sets $F_\bullet \subset X_\bullet$ which contain $v$. If we let $f^*\pi\vert_{|F_\bullet|} \colon f^*\mathcal E(M)\vert_{|F_\bullet|} \to |F_\bullet|$ denote the pullback of $f^*\pi$ along the inclusion $|F_\bullet| \to |X_\bullet|$, 
then as each compact subset of $|X_\bullet|$ lies in the geometric realisation of a finite sub-semi-simplicial set, the map
\[\hocolim_{F_\bullet \in \mathcal{F}} \, \mathrm{hofib}_v(f^*\pi\vert_{|F_\bullet|}) \lra \mathrm{hofib}_v(f^*\pi)\]
is a weak homotopy equivalence. As each $f^*\pi\vert_{|F_\bullet|}$ is a weak quasi-fibration the left-hand side may be replaced with the homotopy colimit of the constant diagram $(f^*\pi)^{-1}(v)$, which shows that $(f^*\pi)^{-1}(v) \to \mathrm{hofib}_v(f^*\pi)$ is a weak homotopy equivalence.
\end{proof}

\subsection{The stable vertical normal bundle}\label{sec:StabNormBundle}

Our goal is to construct a stable $\CAT$ bundle on the total space $\mathcal{E}(M)$ of the universal block bundle $\pi \colon \mathcal{E}(M) \to \mathcal{M}(M)$. We shall focus on the unoriented case for simplicity, but there are no significant changes necessary to treat the oriented case. Our construction will be quite natural once we pull back the universal block bundle to a slightly different, but homotopy equivalent, base. In comparison to the previous section, we shall construct a model for $\mathcal{M}(M)$ which also encodes choices of $\CAT$ normal bundles. This will allow us to essentially follow the argument \cite[Proposition 3.2]{ERW} using this model of the universal block bundle.

\begin{Def}\label{def:MMprime}
If $W \in \mathcal{M}(M)^{n, \epsilon}_p$, an \emph{$\epsilon$-prepared normal $\CAT$ bundle} for $W$ consists of an open neighbourhood $W \subset U \subset \Delta^p \times \bR^n$, a retraction $r \colon U \to W$, and a $\CAT$ $\bR^{n-d}$-bundle atlas $\mathcal{A}$ for $r$. In addition we require that $r$ is a morphism of $p$-block spaces, and that for each $i=0,1,\ldots,p$
\begin{enumerate}[(i)]

\item  $U \cap (\Delta^p_i(\epsilon) \times \bR^n) = (\pi_i(\epsilon) \times \mathrm{Id}_{\bR^n})^{-1}(U\cap(\Delta_i^{p-1} \times \bR^n))$,

\item\label{it:ProdStr} the map $r$ restricted to $U \cap (\Delta^p_i(\epsilon) \times \bR^n)$ commutes with the $i$th barycentric coordinate $t_i$ (which makes the left hand vertical map in the following diagram well defined), and 
\begin{equation*}
\xymatrix{
U \cap (\Delta^p_i(\epsilon) \times \bR^n) \ar[d]_-{r\vert_{U \cap (\Delta^p_i(\epsilon) \times \bR^n)}} \ar[rr]^-{\pi_i(\epsilon) \times \mathrm{Id}_{\bR^n}}& & U \cap (\Delta^{p-1}_i \times \bR^n) \ar[d]^-{r\vert_{U \cap (\Delta^{p-1}_i \times \bR^n)}} \\
W \cap (\Delta^p_i(\epsilon) \times \bR^n) \ar[rr]^-{\pi_i(\epsilon) \times \mathrm{Id}_{\bR^n}}& & W \cap (\Delta^{p-1}_i \times \bR^n)
}
\end{equation*}
is a pullback of $\CAT$ $\bR^{n-d}$-bundles (with the $\CAT$ bundle structure on both sides given by restriction of $\mathcal{A}$).
\end{enumerate}
\end{Def}

\begin{Def}
Let $\mathcal{M}'(M)_\bullet^{\epsilon, n}$ denote the semi-simplicial set with $p$-simplices given by tuples $(W, U, r, \mathcal{A})$ of a $W \in \mathcal{M}(M)_p^{\epsilon, n}$ and an $\epsilon$-prepared normal bundle $(U, r, \mathcal{A})$. The $i$th face map is given by restricting all three pieces of data to $\Delta^{p-1}_i \times \bR^n$.
Again, let $\mathcal M'(M)^n_\bullet = \bigcup_{\epsilon > 0} \mathcal{M}'(M)_\bullet^{\epsilon, n}$. There are maps $\mathcal{M}'(M)_\bullet^{n} \to \mathcal{M}'(M)_\bullet^{n+1}$ given by sending $(W, U, r)$ to $(W, U \times \bR, r \circ \mathrm{proj}_U)$ and we let $\mathcal{M}'(M)_\bullet = \colim\limits_{n \to \infty} \mathcal{M}'(M)_\bullet^{n}$, and $\mathcal{M}'(M) = \vert \mathcal{M}'(M)_\bullet\vert$.
\end{Def}

\begin{Lemma}\label{lem:prep}
The semi-simplicial map $\mathcal{M}'(M)_\bullet \to \mathcal{M}(M)_\bullet$, given by forgetting the bundle data, is a weak homotopy equivalence on geometric realisation.
\end{Lemma}
\begin{proof} 
We shall show that the map has vanishing relative homotopy groups. Our main tool is the relative stable existence and uniqueness theorem for normal $\CAT$ microbundles, and the $\CAT$ microbundle representation theorem. We have explained that $\mathcal{M}(M)_\bullet$ is Kan, and the same argument shows that $\mathcal{M}'(M)_\bullet$ is too, so a relative homotopy class may be described by a submanifold $W \subset \Delta^p \times \bR^n$ such that $W\vert_{\partial \Delta^p}$ comes with a prepared normal $\CAT$ bundle given by $W\vert_{\partial \Delta^p} \subset U_\partial \subset \partial \Delta^p \times \bR^n$, $r_\partial \colon U_\partial \to W\vert_{\partial \Delta^p}$, and $\mathcal{A}_\partial$. In order to show that this relative homotopy class is trivial, it will be sufficient to show that (after perhaps increasing $n$) the prepared normal bundle $(U_\partial, r_\partial, \mathcal{A}_\partial)$ for $W\vert_{\partial \Delta^p}$ is the restriction of a prepared normal bundle for $W$.

For a $\delta>0$ let us write $\Delta^p(\delta) = \cup_{i=0}^p \Delta^p_i(\delta) \subset \Delta^p$. Choose $\epsilon>0$ so that the given data lie in $\mathcal M'(M)^{\epsilon, n} $ or $\mathcal M(M)^{\epsilon, n}$. The product structures given by \cref{def:MM} (\ref{it:Collar}) and \cref{def:MMprime} (\ref{it:ProdStr}) give an extension of $(U_\partial, r_\partial, \mathcal{A}_\partial)$ to a normal $\CAT$ bundle of $W\vert_{\partial W(\epsilon/2)}$, where $\partial W(\epsilon/2) = W \cap (\Delta^p(\epsilon/2) \times \bR^n)$. Furthermore, the submanifold $W\vert_{\Delta^p \setminus \Delta^p(\epsilon)} \subset (\Delta^p \setminus \Delta^p(\epsilon)) \times \bR^n$ has a normal $\CAT$ microbundle (after perhaps increasing $n$) \cite[p.\ 204]{KS}, and this may be represented by a $\CAT$ $\bR^{n-d}$-bundle (by Kister--Mazur \cite{Kister} for $\TOP$, Kuiper--Lashof \cite{KuiperLashof} for $\PL$, and the tubular neighbourhood theorem for $\DIFF$). These yield $\CAT$ normal $\bR^{n-d}$-bundles over the boundary of 
\[W\vert_{\Delta^p(\epsilon) \setminus \Delta^p(\epsilon/2)} \cong W\vert_{\partial \Delta^p} \times [\epsilon/2, \epsilon] \subset \partial \Delta^p \times \bR^n \times [\epsilon/2, \epsilon].\]
By stable uniqueness of $\CAT$ normal microbundles, and of representing $\CAT$ $\bR^{n-d}$-bundles, there is an extension of the $\CAT$ normal $\bR^{n-d}$-bundles over the boundary to the whole of $W\vert_{\Delta^p(\epsilon) \setminus \Delta^p(\epsilon/2)}$. Gluing these three $\CAT$ normal $\bR^{n-d}$-bundles together shows that $(U_\partial, r_\partial, \mathcal{A}_\partial)$ is the restriction of a prepared normal bundle for $W$.
\end{proof}

Let us write $\mathcal{E}'(M)_p^{n} \subset \mathcal{M}'(M)_p^{n} \times \Delta^p \times \bR^n$ for the subspace of those $(W, U, r, \mathcal{A} ; t_0, \ldots, t_p ; x)$ such that $(t_0, \ldots, t_p ; x) \in W$, and $U_p^{n} \subset\mathcal{M}'(M)_p^{n} \times \Delta^p \times \bR^n$ be the subspace of those tuples $(W, U, r, \mathcal{A} ; t_0, \ldots, t_p ; x)$ such that $(t_0, \ldots, t_p ; x) \in U$. We define
\[\mathcal{E}'(M)^{n} := \vert \mathcal{E}(M)^{n}_\bullet \vert := \left(\bigsqcup_{p \geq 0} \mathcal{E}(M)_p^{n} \right)/\sim \quad\quad\quad\quad \vert U^{n}_\bullet \vert := \left(\bigsqcup_{p \geq 0} U_p^{n} \right)/\sim\]
where in both cases $\sim$ is generated by
\[(W, U, r, \mathcal{A}; t_0, \ldots, t_{i-1}, 0, t_{i+1}, \ldots t_p ; x) \sim (d_i(W, U, r, \mathcal{A}); t_0, \ldots, t_{i-1}, t_{i+1}, \ldots t_p ; x).\]
There are maps $r_p^n \colon U^n_p \to \mathcal{E}'(M)_p^{n}$ given by 
\begin{align*}
r_p^n(W, U, r, \mathcal{A} ; t_0, \ldots, t_p ; x) &= (W, U, r, \mathcal{A} ; r(t_0, \ldots, t_p ; x))
\end{align*}
which assemble to a map $r^n\colon \vert U^{n}_\bullet \vert \to \vert \mathcal{E}'(M)^{n}_\bullet \vert$.

\begin{Lemma}
The map $r^n\colon \vert U^{n}_\bullet \vert \to \vert \mathcal{E}'(M)^{n}_\bullet \vert$ has the structure of a $\CAT$ $\bR^{n-d}$-bundle $\nu_{n-d}$, and the restriction of $\nu_{n-d}$ to $\vert \mathcal{E}'(M)^{n-1}_\bullet \vert \subset \vert \mathcal{E}'(M)^{n}_\bullet \vert$ is canonically isomorphic to $\nu_{n-1-d} \times \bR$.
\end{Lemma}
\begin{proof}
A $p$-simplex 
\[\sigma = (W_\sigma, U_\sigma, r_\sigma, \mathcal{A}_\sigma) \in \mathcal{M}'(M)^{n}_p\] 
determines a map $\sigma \colon \Delta^p \to \vert \mathcal{M}'(M)^{n}_\bullet \vert$, so that $\sigma^* \vert \mathcal{E}'(M)^n_\bullet \vert = W_\sigma$. The map $r^n\colon\vert U^{n}_\bullet \vert \to \vert \mathcal{E}'(M)^{n}_\bullet \vert$ pulled back to this is precisely $r_\sigma \colon U_\sigma \to W_\sigma$, which is a locally trivial $\CAT$ $\bR^{n-d}$-bundle via the atlas $\mathcal{A}_\sigma$. 
Now let 
\[\vert \mathcal{E}'(M)^{n}_\bullet \vert^{(k)} = \left(\bigsqcup_{p=0}^k \mathcal{E}'(M)_p^n\right)/\sim\] 
denote the $k$-skeleton, similarly $\vert U^{n}_\bullet \vert^{(k)}$, and suppose given a $\CAT$ atlas $\mathcal{A}^{(k)}$ for $r^{(k)} \colon \vert U^{n}_\bullet \vert^{(k)} \to \vert \mathcal{E}'(M)^{n}_\bullet \vert^{(k)}$ which over each simplex $(W, U, r, \mathcal{A})$ restricts to the atlas $\mathcal{A}$ for $r \colon U \to W$. For each $(k+1)$-simplex 
\[\sigma = (W_\sigma, U_\sigma, r_\sigma, \mathcal{A}_\sigma) \in \mathcal{M}'(M)^{n}_{k+1}\]
there is an $\epsilon > 0$ such that for each $i=0,1,\ldots, p$ we have
\[W_\sigma \cap (\Delta^p_i(\epsilon) \times \bR^n) = \pi_i(\epsilon)^{-1}(W_\sigma\cap(\Delta_i^{p-1} \times \bR^n))\]
and
\[U_\sigma \cap (\Delta^p_i(\epsilon) \times \bR^n) = \pi_i(\epsilon)^{-1}(U_\sigma\cap(\Delta_i^{p-1} \times \bR^n))\]
and on this set $r$ commutes with the $i$th barycentric coordinate $t_i$ and satisfies $\pi_i(\epsilon) \circ r_\sigma = r_\sigma \circ \pi_i(\epsilon)$. In particular, the inclusion $\partial W_\sigma \to \partial^\epsilon W_\sigma$, where 
\begin{align*}
\partial W_\sigma =& \bigcup_{i=0}^p W_\sigma \cap (\Delta_i^{p-1} \times \bR^n) \quad\quad\text{ and }\quad\quad
\partial^\epsilon W_\sigma = \bigcup_{i=0}^p W_\sigma \cap (\Delta^p_i(\epsilon) \times \bR^n),
\end{align*}
has a retraction $\rho_\sigma$ such that $U_\sigma\vert_{\partial^\epsilon W_\sigma} \cong \rho_\sigma^* U_\sigma\vert_{\partial W_\sigma}$ as $\CAT$ $\bR^{n-d}$-bundles. Thus $\rho_\sigma^*(\mathcal{A}^{(k)})$ gives a $\CAT$ atlas over $\partial^\epsilon W_\sigma$ which is compatible with $\mathcal{A}_\sigma$. This shows that there is an atlas $\mathcal{A}^{(k+1)}$ for $r^{(k+1)} \colon \vert U^{n}_\bullet \vert^{(k+1)} \to \vert \mathcal{E}'(M)^{n}_\bullet \vert^{(k+1)}$ extending the atlas $\mathcal{A}^{(k)}$ for $r^{(k)}$.

Gluing together the sets $\partial^\epsilon W_\sigma$ for all $(k+1)$-simplices $\sigma$ gives an open subset 
\[V^{(k)} \subseteq \vert \mathcal{E}'(M)^{n}_\bullet \vert^{(k+1)}\] 
containing $\vert \mathcal{E}(M)^{n}_\bullet \vert^{(k)}$. The retractions $\rho_\sigma$ glue together to a retraction 
\[\rho^{(k)} \colon V^{(k)} \lra \vert \mathcal{E}'(M)^{n}_\bullet \vert^{(k)}\]
 such that 
 \[\vert U^{N}_\bullet \vert^{(k+1)} \vert_{V^{(k)}} \cong (\rho^{(k)})^* \vert U^{N}_\bullet \vert^{(k)}\] 
 as $\CAT$ $\bR^{n-d}$-bundles. A point $x \in \vert \mathcal{E}'(M)^{n}_\bullet \vert^{(k)}$ has an open neighbourhood 
\[V_x = V^{(k)} \cup (\rho^{(k+1)})^{-1}(V^{(k)}) \cup (\rho^{(k+1)} \circ \rho^{(k+2)})^{-1}(V^{(k)}) \cup \cdots \subset \vert \mathcal{E}'(M)^{n}_\bullet \vert\]
which retracts to $\vert \mathcal{E}'(M)^{n}_\bullet \vert^{(k)}$ via 
\[\rho_x = \rho^{(k)} \cup (\rho^{(k)} \circ \rho^{(k+1)}) \cup (\rho^{(k)} \circ \rho^{(k+1)} \circ \rho^{(k+2)}) \cup \cdots, \] 
and $\vert U^{n}_\bullet \vert\vert_{V_x} \cong \rho_x^* \vert U^{n}_\bullet \vert^{(k)}$ as $\CAT$ $\bR^{n-d}$-bundles. This proves the first part of the lemma; the second part is immediate from the formula for the map $\mathcal{E}'(M)^{n-1}_\bullet \to \mathcal{E}'(M)^n_\bullet$.
\end{proof}

Note that $\mathcal{E}'(M)^{n} = \vert \mathcal{E}'(M)^{n}_\bullet \vert$ is paracompact by a similar argument to that which shows that a cell complex is paracompact, and hence the $\CAT$ $\bR^{n-d}$-bundle $\nu_{n-d}$ is numerable, so is classified by a map $\nu_{n-d} \colon  \mathcal{E}'(M)^{n} \to \B\CAT(n-d)$. We thus obtain a diagram
\begin{equation*}
\xymatrix{
 \ar[r] & \mathcal{E}'(M)^{n} \ar[d]^{\nu_{n-d}} \ar[r] & \mathcal{E}'(M)^{n+1} \ar[d]^{\nu_{n+1-d}} \ar[r] & \mathcal{E}'(M)^{n+2} \ar[d]^{\nu_{n+2-d}} \ar[r] & {}\\
\ar[r]& \B\CAT(n-d) \ar[r]& \B\CAT(n+1-d) \ar[r]& \B\CAT(n+2-d) \ar[r] &
}
\end{equation*} 
in which each square homotopy commutes up to a preferred homotopy class of homotopies, and so taking (homotopy) colimits we obtain a map $\nu_v(\pi') \colon \mathcal{E}'(M) \to \B\CAT$. Now, the square
\begin{equation*}
\xymatrix{
\mathcal{E}'(M) \ar[r] \ar[d]^-{\pi'}& \mathcal{E}(M) \ar[d]^-{\pi}\\
\mathcal{M}'(M) \ar[r] & \mathcal{M}(M)
}
\end{equation*}
is homotopy cartesian by \cref{prop:quasifib} so the top map is a weak equivalence. Thus we may transfer the map $\nu_v(\pi')$ to a map 
\[\nu_v(\pi) \colon \mathcal{E}(M) \lra \B\CAT\]
classifying what we shall call the \emph{$\CAT$ stable vertical normal bundle}. We call its stable inverse the \emph{$\CAT$ stable vertical tangent bundle}, and denote it $T_v^s(\pi)$.

\subsection{Comparisons}\label{comparison bundle} Let us finally compare this definition with both the usual vertical tangent bundle of a fibre bundle, and the stable bundle constructed in \cite{ERW}.

\subsubsection{The vertical tangent bundles of fibre bundles}\label{sec:luyoou}

If a $\CAT$ $M$-fibre bundle is considered as a $\CAT$ $M$ block bundle, then the stabilisation of its $\CAT$ vertical tangent bundle agrees with the $\CAT$ stable vertical tangent bundle we have constructed.

The simplest way to prove this is to work universally, and produce a model $\mathcal B(M)$ for $\B\CAT(M)$ akin to $\mathcal{M}(M)$ by realising the semi simplicial set with $p$-simplices the locally flat $\CAT$ submanifolds $W \subset \Delta^p \times \R^\infty$ so that the map to the first factor is a $\CAT$ $M$-bundle. Just as in the case of block bundles there is a version $\mathcal B'(M)$ of this construction where manifolds are equipped with choices of tubular neighbourhoods $(U,r,\mathcal A)$ as before, where one additionally insists that the map $r \colon U \to W$ is fibrewise over $\Delta^p$. This space $\mathcal B'(M)$ has a forgetful map to $\mathcal M'(M)$, and the pullback of $\pi' : \mathcal E'(M) \to \mathcal M'(M)$ to $\mathcal B'(M)$ gives a universal $M$-fibre bundle $\mathcal F'(M) \to \mathcal B'(M)$, to which the stable vertical normal bundle $\nu_v(\pi')$ can be pulled back. The vertical tangent bundle of $\mathcal F'(M) \to \mathcal B'(M)$ is a stable inverse to this, by construction.

\begin{Rmk}
This comparison proves that the stable vertical tangent bundle of a topological manifold bundle only depends on its underlying block bundle and thus our constructions recover \cite[Theorem G]{BFJ}: Their \emph{strong Borel conjecture} is well-known to imply our block Borel conjecture (we will explain this in the proof of \cref{FJ vs BBC}) and therefore that fibre homotopy equivalent $M$-(block-)bundles are equivalent as block bundles, so must have isomorphic stable vertical tangent bundles.
\end{Rmk}

\subsubsection{The stable vertical tangent bundle of \cite{ERW}}
The authors of that paper considered a smooth block bundle $(p \colon E \to \vert K \vert, \mathcal{A})$ with base the geometric realisation of a finite simplicial complex $K$. In \cite[Proposition 3.2]{ERW} they constructed a stable vertical tangent bundle by choosing embeddings $e\colon E \to \vert K \vert \times \bR^n$ and $a \colon \vert K \vert \to \bR^k$ satisfying certain properties, and hence constructing a continuous map $E \to Gr_{d+k}(\bR^{n+k})$: the $(d+k)$-dimensional vector bundle classified by this map is called $t_{E,e,a}$, and is the stable vertical tangent bundle; the $(n-d)$-dimensional vector bundle classified by this map is called $n_{E,e,a}$, and is the stable vertical normal bundle. 

If the classifying map for a smooth block bundle $(p \colon E \to \vert K \vert, \mathcal{A})$  is factored up to homotopy as $\vert K \vert \to \vert \mathcal{M}(M)_\bullet^{n} \vert \to \vert \mathcal{M}(M)_\bullet\vert$, then the block bundle is concordant to a $(p' \colon E' \to \vert K \vert, \mathcal{A}')$ which comes equipped with an embedding $e' \colon E' \to \vert K \vert \times \bR^n$ a neighbourhood $E' \subset U' \subset \vert K \vert \times \bR^n$, and a retraction $r' \colon U' \to E'$ which has the structure of a smooth $\bR^{n-d}$-bundle. This yields a $(n-d)$-dimensional vector bundle on $E'$, and this is isomorphic to $n_{E',e',a'}$ for any choice of $a' \colon \vert K \vert \to \bR^k$. In particular, the associated $t_{E,e,a}$ is stably isomorphic to the stable vertical tangent bundle constructed here.

\subsubsection{Stable vertical tangent bundles of block bundles over manifolds}

Given a block bundle over a triangulated manifold, one may describe its stable vertical tangent bundle in terms of the tangent bundles of the base and total space, as follows.

\begin{Lemma}\label{compb}
Let $\vert K \vert \stackrel \cong \lra B$ be a $\PL$ triangulation of a $\CAT$ manifold (compatible in the smooth or piecewise linear cases), and $(p \colon E \to \vert K \vert, \mathcal{A})$ be a $\CAT$ block bundle. Then $E$ has the structure of a $\CAT$ manifold, and the stable vertical $\CAT$ tangent bundle is equivalent to $TE - p^*TB$. 
\end{Lemma}
\begin{proof}
Let us first show that $E$ inherits a $\CAT$ manifold structure. The stars $\mathrm{St}(v) \subset \vert K \vert$ of vertices $v \in K$ have interiors which form an open cover of $\vert K \vert$, so their preimages $p^{-1}(\mathrm{St}(v))$ have interiors which form an open cover of $E$ and hence it is enough to give (compatible) $\CAT$ manifold structures to these. We have
\[p^{-1}(\mathrm{St}(v)) = \bigcup_{\sigma \ni v} W_\sigma\]
where $W_\sigma$ is the block over $\sigma$. There are $\CAT$ isomorphisms $W_\sigma \cong \sigma \times M$. As mentioned in \cref{Remark Kan}, the semi-simplicial group $\widetilde{\CAT}(M)_\bullet$ is Kan so that we may choose such $\CAT$ isomorphisms in increasing order of $\dim(\sigma)$, extending those which have already been chosen on faces of $\sigma$ (we use here that all simplices of $\mathrm{St}(v)$ have a free face). This gives a block $\CAT$ isomorphism $p^{-1}(\mathrm{St}(v)) \cong \mathrm{St}(v) \times M$, and hence induces a $\CAT$ manifold structure on $p^{-1}(\mathrm{St}(v))$.

By \cref{lem:prep} we may suppose that  $(p \colon E \to \vert K \vert, \mathcal{A})$ is classified by a map to some $\vert \mathcal{M}(M)^{n}_\bullet\vert$, so we have a neighbourhood $E \subset U \subset \vert K \vert \times \bR^n$ and a retraction $r \colon U \to E$ equipped with the structure of a $\CAT$ $\bR^{n-d}$-bundle. By the same argument as above, $U$ has a $\CAT$ manifold structure making it an open submanifold of $\vert K \vert \times \bR^n$. By the uniqueness theorem for stable normal (micro)bundles \cite[p.\ 204]{KS}, this must be isomorphic to the normal bundle of $E \subset \vert K \vert \times \bR^n$, which is stably $TE - p^*TB$.
\end{proof}

\begin{Rmk}
Let us comment on a relation to \cite[Theorem 1.1]{BM}. Phrased in the language of classifying spaces we construct for any of $\CAT \in \{\DIFF, \mathrm{PL}, \TOP\}$ a dashed arrow in the following commutative diagram:
\[\xymatrix{\B\CAT(M) \ar[d]\ar[r] & \B\mathcal G^\CAT(T^sM) \ar[d]\\
            \B\widetilde{\CAT}(M) \ar[r]\ar@{-->}[ru] & \B\mathcal G(M)}\]
Here $\G^\CAT(T^s M)$ denotes the space of homotopy self equivalences of $M$ covered by a $\CAT$-bundle map of the stable tangent bundle of $M$ and all solid arrows are the evident forgetful maps. This should be compared to \cite[Section 4]{BM}, where Berglund and Madsen construct a similar map. Since the space $\B\G^\CAT(T^sM)$ classifies fibrations with fibre $M$ equipped with a stable $\CAT$-bundle on the total space which restricts to the stable tangent bundle on each fibre, by construction, the tautological classes indeed give rise to classes in $H^*(\B\G^\CAT(T^sM);R)$. We will not make use of this fact.
\end{Rmk}

\section{An Euler class for fibrations with Poincar\'e fibre}\label{sec:Euler}

In \cite[Section 2]{RW} Randal-Williams constructs a fibrewise Euler class for a fibration $p \colon E \to B$ in which $B$ is a finite complex, the fibre $F$ is an oriented Poincar\'e duality space of formal dimension $d$, and the fibration is oriented in the sense that the monodromy action of $\pi_1(B)$ on $H_d(F;\Z)$ is trivial. However, the line of argument used essentially that $B$ is a finite cell complex, so cannot be used to obtain an Euler class for the fibration 
\[F \lra \B\G^+_*(F) \lra \B\G^+(F),\]
which is easily checked to be the universal one. To define an Euler class also when the base $B$ is not necessarily a finite complex -- e.g.\ $\B\G^+(F)$ and $\B\TOPtw{\mathstrut}^{\>\!+}(M)$ -- we shall give a different argument using parametrised stable homotopy theory.

\subsection{The fibrewise Euler class}\label{sec:ConstEul}
To motivate our construction, let us recall one definition of the Euler class of a Poincar\'e duality space $F$. If $D_F \colon H^{*}(F) \to H_{d-*}(F)$ is the Poincar\'e duality map for $F$, and $\Delta \colon F \to F \times F$ is the diagonal map, then the Euler class of $F$ may be described as $e(F) = \Delta^* D_{F \times F}\Delta_*D_F(1)$. We will mimic this definition for a fibration $p \colon E \to B$ with Poincar\'e fibre $F$ using parametrised (co)homology.

Consider the category $\Sp_B$ of spectra parametrised over $B$. Our results are essentially model independent but for definiteness' sake we shall take \cite{MS} as our primary reference. Suppressing usual subscripts to declutter the notation, as no other base will be considered in this section, let $\mathrm H\Z$ and $\Ss$ denote the trivially parametrised Eilenberg--Mac Lane and sphere spectrum, respectively. Similarly, we have dropped the notation for a disjoint base point or section. We have the diagonal map $\Delta \colon  E \to E \times_B E$ and
we claim that Poincar{\'e} duality for $F$, and the orientability hypothesis, yield a \emph{fibrewise Poincar\'e duality} equivalence
\[D_E^{fw} \colon \Sigma^d F_B(E , \mathrm H\Z) \overset{\simeq} \lra E \wedge_B \mathrm H\Z  \]
of $\mathrm H\Z$-modules, and similarly for $E \times_B E$. Here $-\wedge_B-$ denotes the derived fibrewise smash product and $F_B(-,-)$ denotes the derived fibrewise mapping spectrum. The point-set versions of these constructions are explained in \cite[Section 11]{MS} and their derived versions are established in \cite[Sections 12.6 \& 13.1]{MS}. Granted this claim, the diagram
\[\xymatrix{\mathbb S^d \ar[r]^-{\Sigma^d 1} & \Sigma^d F_B(E , \mathrm H\Z) \ar[r]^-{D_E^{fw}}_-\simeq &  E \wedge_B \mathrm H\Z \ar[rr]^-{\Delta \wedge_B \mathrm H\Z} &&  (E \times_B E) \wedge_B \mathrm H\Z}\]
\vspace{-4ex}
\[\xymatrix{ && \ar[ll]_-{D_{E \times_B E}^{fw}}^-\simeq \Sigma^{2d} F_B(E \times_B E, \mathrm H\Z) \ar[rr]^-{F_B(\Delta, \mathrm H\Z)} & &\Sigma^{2d} F_B(E , \mathrm H\Z)}\] 
 represents a well-defined element
 \[\e(p) := \Delta^* (D_{E \times_B E}^{fw})^{-1} \Delta_* D_E^{fw}(1) \in [\mathbb S^{-d}, F_B(E , \mathrm H\Z)]_B \cong [\mathbb S^{-d} \wedge_B E, \mathrm H\Z]_B \cong H^d(E;\Z).\] 
  See \cite[Theorem 5.6]{GGRW} for a related discussion.
 \begin{Def}\label{def euler}
 The class $\e(p) \in H^d(E;\Z)$ so constructed is the \emph{fibrewise Euler class} of the oriented fibration $p\colon E\to B$.
 \end{Def}

It remains to establish the equivalence $D_E^{fw}$. As in ordinary Poincar\'e duality, it will be given by cap product with a fundamental class. Consider the parametrised Atiyah--Hirzebruch spectral sequence
\[H^i(B; (E\wedge_B \mathrm H\Z)^j) \Longrightarrow (E\wedge_B \mathrm H\Z)^{i+j}(B)\]
based on the parametrised spectrum $E\wedge_B \mathrm H\Z$, compare e.g.\ \cite[Theorem 20.4.1]{MS} (with $X= B$ and $J= E \wedge_B \mathrm H\Z$ and $\mathscr L^*(B,E \wedge_B \mathrm H\Z)$
abbreviated to $(E\wedge_B \mathrm H\Z)^*$).
The fundamental classes 
\[[E_b] \in H_d(E_b;\Z) \cong H^0(b; (E_b \wedge \mathrm H\Z)^{-d})\] 
of the fibres $E_b$ assemble to a class $[E]_B \in H^0(B;(E \wedge_B \mathrm H\Z)^{-d})$ on the second page by the orientation hypothesis. 
The spectral sequence is concentrated in rows $-d, \dots, 0$ and positive columns, so the remaining groups on the diagonal $i+j=-d$ are zero, and $[E]_B$ is a permanent cycle for degree reasons as well, so abusing notation we obtain a unique class 
\[[E]_B \in (E\wedge_B \mathrm H\Z)^{-d}(B) = [\mathbb S^{d}, E \wedge_B \mathrm H\Z]_B.\]
 We may thus form the cap product with $[E]_B$, that is the map
\[\xymatrix{\mathbb S^{d} \wedge_B F_B(E, \mathrm H\Z) \ar[r]^-{[E]_B \wedge \mathrm{Id}} & (E \wedge_B \mathrm H\Z) \wedge_B F_B(E, \mathrm H\Z)}\]
\vspace{-4ex}
\[\xymatrix{\ar[r]^-{\Delta} & E \wedge_B E \wedge_B \mathrm H\Z \wedge_B F_B(E, \mathrm H\Z)
\ar[r]^-{ev} & E \wedge_B \mathrm H\Z \wedge_B \mathrm H\Z \ar[r] & E \wedge_B \mathrm H\Z,}\]
which is the sought after map $D_E^{fw}$; that it is an equivalence may be checked on fibres, where it reduces to ordinary Poincar{\'e} duality.

\begin{Rmk}
The above construction of a fibrewise Euler class clearly works for a general ring spectrum $R$, whenever the fibration admits a fibrewise fundamental class $[E]_B \in [\mathbb S^d, E \wedge_B R]$. However, even if there is a class $[E]_B \in H^0(B; (E \wedge_B R)^{-d})$ restricting to an $R$-homology fundamental class of each fibre $E_b$, it need not come from a class in $[\mathbb S^d, E \wedge_B R]$,  unless $R$ is co-connective.

For an explicit counterexample (in the spirit of this paper), consider the ring spectrum $R = \mathbb S[\tfrac{1}{2}]$ and an oriented surface bundle $\Sigma \to E \overset{\pi}\to B$. The Hurewicz map
\[ \mathbb S[\tfrac{1}{2}]_2(\Sigma) \lra H_2 (\Sigma; \mathbb Z[\tfrac{1}{2}])\]
is an isomorphism, so the $\mathbb Z[\tfrac{1}{2}]$-homology fundamental classes of the fibres of $\pi$ yield a class $[E]_B \in H^0\left(B; (E \wedge_B \mathbb S[\tfrac{1}{2}])^{-2}\right)$ restricting to a $\mathbb S[\tfrac{1}{2}]$-homology fundamental class of each fibre. If this lifted to $[E]_B \in \left[\mathbb S^2, E \wedge_B \mathbb S[\tfrac{1}{2}]\right]$ then one could follow the construction above to form an Euler class
$$e^{fw}(\pi) \in \mathbb S[\tfrac{1}{2}]^2(E)$$
which under the Hurewicz map gives the ordinary Euler class $e(T_v(\pi)) \in H^2\left(E;\mathbb Z[\tfrac{1}{2}]\right)$ of the vertical tangent bundle. But for any odd prime $p$, using that $e^p = \mathcal{P}^{1}(e)$ and that $\mathcal{P}^1$ is trivial on $H^*\left(\mathbb S[\tfrac{1}{2}];\mathbb F_p\right)$, this would mean that $e^p = 0 \in H^{2p}(E;\mathbb F_p)$ and hence that $\kappa_{e^p}=0 \in H^{2p-2}(B;\mathbb F_p)$. By taking the genus of $\Sigma$ to be large enough, and $p=3$, this contradicts \cite[Theorem 1.2]{GMT}.
\end{Rmk}

\subsection{Comparisons} Again we compare our construction to both the classical case and the definition of \cite{RW}.

\subsubsection{The Euler class of the vertical tangent bundle}\label{sec:98yyohh}
Suppose that $p \colon E \to B$ is an oriented topological fibre bundle with fibre a $d$-dimensional manifold $M$, with $B$ a CW-complex. The data $(\pi_1 \colon E \times_B E \to E, \Delta \colon E \to E \times_B E)$ defines the vertical tangent topological microbundle $T_v(p)$ over $E$. As $B$ is a CW-complex it follows that $E$ is paracompact, so by \cite{Holm} it contains a Euclidean $\bR^d$-bundle, i.e.\ there is an open neighbourhood $E \overset{s}\hookrightarrow U \subset E \times_B E$ with a projection $r \colon U \to E$ over $B$ which is a Euclidean $\bR^d$-bundle. Writing $U^+_B$ for the fibrewise 1-point compactification, there is a fibrewise collapse map
\[c \colon E \times_B E \lra U^+_B.\]
The composition 
\[E \overset{\Delta}\lra E \times_B E \overset{c}\lra U^+_B \overset{q}\lra U^+_B/B  = \mathrm{Th}(U)\]
pulls back the Thom class $u \in H^d(\mathrm{Th}(U) ;\mathbb{Z})$ to the Euler class $e(T_v(p))$ of $T_v(p)$.

To compare this with the definition above, consider the map 
\[d \colon U^+_B \lra E \wedge_B U^+_B\]
induced by the diagonal map of $U$, which fits into a commutative diagram
\begin{equation*}
\xymatrix{
(E \times_B E) \wedge_B \mathrm H \Z \ar[rr]^-{\Delta_{E \times_B E} \wedge \mathrm H \Z} \ar[d]^{c \wedge \mathrm H \Z}& & (E \times_B E \times_B E \times_B E) \wedge_B \mathrm H \Z \ar[dd]^-{E \times E \times c \wedge \mathrm H \Z}\\
U^+_B \wedge_B \mathrm H \Z \ar[d]^-{d \wedge \mathrm H \Z}\\
E \wedge_B U^+_B \wedge_B \mathrm H \Z \ar[rr]^{\Delta_E \wedge U^+_B \wedge \mathrm H \Z} \ar[d]^-{E \wedge q^*u} & & (E \times_B E) \wedge_B U^+_B \wedge_B \mathrm H \Z \ar[d]^-{E \wedge q^*u}\\
E \wedge_B  \mathrm H \Z \ar[rr]^{\Delta_E \wedge \mathrm H\Z} & & (E \times_B E) \wedge_B  \mathrm H \Z.
}
\end{equation*}
Precomposing this with the map 
\[\xymatrix{\mathbb S^{2d} \ar[rr]^-{[E \times_B E]_B} && (E \times_B E) \wedge_B \mathrm H \Z}\]
by definition gives $[E \times_B E]_B \frown c^*q^*u$ along the top. Under the equivalence $E \to U$ we have 
\[c_*[E \times_B E]_B \frown q^*u = [E]_B \in [\mathbb S^d,  E \wedge_B \mathrm H\Z]_B;\] by definition of $[E]_B$ this can be checked by restriction to a single fibre, where it reduces to Thom's description of Poincar\'e duals of submanifolds. Composition along the bottom is therefore $\Delta_*([E]_B)$. Hence $c^*q^*u = (D_{E \times_B E}^{fw})^{-1}\Delta_*([E]_B) \in H^d(E;\Z)$, and so
\[e(T_v(p)) = s^*q^*(u) = \Delta^* c^*q^*(u) = \Delta^* (D_{E \times_B E}^{fw})^{-1}\Delta_*D_E^{fw}(1) = \e(p).\]

\subsubsection{The Euler class of \cite{RW}}

The construction in \cite[Section 2]{RW} follows the proof of the `Fibre Inclusion Theorem' of Casson--Gottlieb \cite{CG}: by embedding $B$ into some $\R^n$, taking a regular neighbourhood, and doubling it, we may find an embedding $i\colon B \to B'$ into an oriented smooth $n$-manifold and a retraction $r \colon B' \to B$. Then $E' :=r^*E \to B'$ is a fibration with oriented Poincar{\'e} base and fibre, so $E'$ is also oriented Poincar{\'e}, by \cite{Gottlieb2}. Let us write $D_{E'} \colon H^*(E') \to H_{n+d-*}(E')$ for the Poincar{\'e} duality isomorphism. Similarly, $E' \times_{B'} E'$ is Poincar{\'e} with duality isomorphism $D_{E' \times_B E'}$, and using the diagonal map $\Delta \colon E' \to E' \times_{B'} E'$ we can form 
\[e(E') := \Delta^* D_{E' \times_B E'}^{-1}\Delta_*D_{E'}(1) \in H^d(E';\Z)\]
The {Euler class} $e(E) \in H^d(E;\Z)$ is then defined by restriction along $E \to E'$.

The key step in comparing this definition lies in the comparison between the usual Poincar\'e duality of $E$ and its fibrewise Poincar\'e duality. For this we need to make use of the notion of Costenoble--Waner duality, the parametrised (as opposed to fibrewise) version of Spanier-Whitehead duality, see \cite[Section 2.9]{CW} and \cite[Chapter 18]{MS}. Suppose then $p \colon E \to B$ is a fibration with $n$-dimensional oriented manifold base, and write $r \colon B \to *$ for the constant map. The associated pull back functor $r^* \colon \mathrm{Ho}\Sp \rightarrow \mathrm{Ho}\Sp_B$ admits both a left and a right adjoint denoted
\[r_! \colon \mathrm{Ho}\Sp_B \longrightarrow \mathrm{Ho}\Sp\quad \text{and} \quad r_* \colon \mathrm{Ho}\Sp_B \longrightarrow \mathrm{Ho}\Sp,\]
respectively, see \cite[Proposition 12.6.7 \& Theorem 13.1.18]{MS}. They come with canonical identifications
\[r_!(E \wedge_B \mathrm H\Z) \simeq E \wedge \mathrm H\Z \quad\quad \text{and} \quad\quad r_*(F_B(E,\mathrm H\Z)) \simeq F(E,\mathrm H\Z).\]
Let us write $\nu_B \in \Sp_{B}$ for the Spivak normal fibration of $B$, regarded as a parametrised spectrum over $B$ with fibre $\mathbb S^{-n}$. 
By parametrised Atiyah duality (\cite[Theorem 18.6.1]{MS}) $\nu_B$ is the Costenoble--Waner dual of the sphere spectrum over $B$, so by \cite[Proposition 18.1.5]{MS} (with $J = E \wedge_B \mathrm H\Z$) we have an equivalence of spectra
\[\mu \colon r_!(E \wedge_B \mathrm H\Z \wedge_B \nu_B) \longrightarrow r_*(E \wedge_B \mathrm H\Z).\]
We have a second equivalence $\mathrm{th} \colon \mathrm H\Z \wedge_B \nu_B \rightarrow \mathrm H\Z \wedge_B \mathbb S^{-n}$ coming from the orientation of $B$ (which becomes the Thom isomorphism upon applying $r_!$). 
Under the assumption that the fibres of $p$ are coherently oriented Poincar{\'e} complexes of dimension $d$ we have a third equivalence $D_E^{fw} \colon \Sigma^dF_B(E,\mathrm H\Z)\rightarrow E \wedge_B \mathrm H\Z$ from \cref{sec:ConstEul}, and 
it is a tedious but straightforward exercise to check that the diagram of equivalences
\[\xymatrix{
\Sigma^{d+n} r_*(F_B(E,\mathrm H\Z)) \ar[d]^{\Sigma^n r_*(D_E^{fw})} \ar[rr]^{D_E}       &&  r_!(E \wedge_B \mathrm H\Z) \\            
            \Sigma^n r_*(E \wedge_B \mathrm H\Z)      &&  \Sigma^n r_!(E \wedge_B \mathrm H\Z \wedge_B \nu_B)\ar[ll]_{\Sigma^n \mu} \ar[u]^{\Sigma^n \mathrm{id}_E \wedge \mathrm{th}} }\]
commutes, where $D_E\colon \Sigma^{d+n}F(E,\mathrm H\Z) \rightarrow E \wedge \mathrm H\Z$ is ordinary Poincar{\'e} duality for $E$. The comparison of Euler classes now follows by splicing this diagram with the analogous one for $E \times_B E \rightarrow B$.
%

\section{Tautological characteristic classes of block bundles}\label{sec3}

In the rest of the paper we shall be interested in \emph{oriented} block bundles. That is, we will assume that $M$ is oriented, and consider block bundles $(p \colon E \to \vert K \vert, \mathcal{A})$ for which the transition maps are orientation preserving. These are classified by analogous spaces 
\[\B{\widetilde{\CAT}}^+(M) \simeq \mathcal{M}^{+}(M),\] 
where the $p$-simplices of $\mathcal{M}^{+}(M)_\bullet$ are \emph{oriented} submanifolds $W \subset \Delta^p \times \bR^n$ which are $p$-block $\CAT^+$ isomorphic to $\Delta^p \times M$. Forgetting the orientation defines a map $f \colon \mathcal{M}^+(M)\to \mathcal{M}(M)$, which defines the universal oriented block bundle $\mathcal{E}^+(M) = f^*\mathcal{E}(M)$ with projection
\[\pi \colon \mathcal E^+(M) \lra \mathcal{M}^+(M)\]
for which we will now define tautological classes. Again the case of interest for us is that of topological block bundles, but our methods work just as well in the smooth and piecewise linear categories, so we work in that generality.

\subsection{The tautological classes}
By \cref{prop:quasifib}, the map $\pi$ is a weak quasi-fibration, i.e.\ $\pi^{-1}(v) \to \mathrm{hofib}_v(\pi)$ is a weak homotopy equivalence for any vertex $v \in \mathcal{M}^+(M)_0$. As $\pi^{-1}(v) \cong M$, the Serre spectral sequence for the map $\pi$ (replaced by a fibration) takes the form
\[E_2^{p,q} = H^p(\mathcal M^+(M); \mathcal{H}^q(M;R)) \Longrightarrow H^{p+q}(\mathcal E^+(M);R).\]
Since the block bundle is oriented, the local system $\mathcal H^d(M;R)$ is trivialised for any ring $R$ and so this spectral sequence defines (see \cite[Section 8]{BorHir}) a Gysin homomorphism
\[\pi_! \colon H^{k}(\mathcal E^+(M);R) \lra H^{k-d}(\mathcal{M}^+(M);R).\] 
The stable vertical tangent bundle constructed in Section \ref{sec:StabNormBundle}, together with the fibrewise Euler class constructed in Section \ref{sec:Euler}, give a map 
\[(T_v^s\mathcal (\pi), \e(\pi))\colon \mathcal E^{+}(M) \lra \B\mathrm S\CAT \times K(\mathbb Z,d)\]
 for any $d$-dimensional $\CAT$ manifold $M$. Using the equivalence $H^*(\mathcal M^+(M); R) \cong H^*(\B\widetilde\CAT{\mathstrut}^{\>\!+}(M);R)$ discussed in Section \ref{bla} we obtain:

\begin{Def}\label{def kappa}
The \emph{universal tautological characteristic classes}
\[\kappa_c(M) := \pi_!((T_v^s\mathcal (\pi), \e(\pi))^*(c))\]
define a homomorphism
\[ \kappa_-(M) \colon H^k(\B\mathrm S\CAT \times K(\mathbb Z,d);R) \lra H^{k-d}(\B\widetilde\CAT{\mathstrut}^{\>\!+}(M);R).\]
\end{Def}

\subsection{Comparisons}
These classes agree with the classes defined in \cite{ERW} and also restrict to the classical tautological classes for the universal smooth fibre bundle. We record this explicitly in the following propositions.
\begin{Prop}\label{comm}
The square
\begin{equation*}\label{eq:CommSq}
\xymatrix{
H^*(\BSTOP \times K(\mathbb{Z},d);R) \ar[r]^-{\kappa_-(M)} \ar[d]& H^{*-d}(\BTOPtw{\mathstrut}^{\>\!+}(M);R) \ar[d]\\
H^*(\BSO(d);R) \ar[r]^-{\kappa_-(M)} & H^{*-d}(\B\DIFF^+(M);R)
}
\end{equation*}
commutes.
\end{Prop}

The implications of this statement depend on the coefficient ring $R$, mostly due to the fact that relevant properties of the left vertical map depend on the choice of coefficients: The work of Kirby--Siebenmann (specifically \cite[p.\ 200]{KS}) implies that the map $\BSO \to \BSTOP$ is a rational equivalence and thus the left vertical map in the diagram is a surjection when $R$ is $\Q$. Therefore, all \emph{rational} tautological classes in $H^*(\B\DIFF^+(M);\Q)$ are in the image of the upper composition. The latter also holds for $R = \mathbb Z/2$ by Thom's description of the Stiefel-Whitney classes. As mentioned in the introduction this is not true for $R = \Z$, but see \cref{intres}. 

If $M$ satisfies the block Borel conjecture then fibre homotopy equivalent smooth fibre bundles with fibre $M$ are in fact concordant as topological block bundles, and so the above diagram then shows that they have the same rational tautological classes; this recovers \cite[Corollary G.1]{BFJ}. If $M$ is a nonpositively curved manifold then it does satisfy the block Borel conjecture (by \cref{FJ vs BBC} as its fundamental group satisfies the Farrell--Jones conjecture by \cite{FJ2}), so the above applies; this generalises \cite[Corollary C.1]{BFJ}. 

\begin{Prop}\label{comm2}
Under the maps 
\[\B\TOP^+(M) \lra \B\TOPtw{\mathstrut}^{\>\!+}(M) \quad \text{    and    } \quad  \B\DIFFtw{\mathstrut}^{\>\!+}(M) \lra \B\TOPtw{\mathstrut}^{\>\!+}(M)\] 
the tautological classes just defined restrict to those of Ebert and Randal-Williams.
\end{Prop}

\begin{proof}[Proof of Propositions \ref{comm} \& \ref{comm2}]
\cref{comm} follows immediately from Sections \ref{sec:luyoou} and \ref{sec:98yyohh}, which together say that the stable vertical tangent microbundle and Euler class of a smooth fibre bundle agree with the objects of the same names we have associated to the corresponding $\TOP$ block bundle. 

For \cref{comm2} let us first consider the easier case of block diffeomorphisms. As discussed in \cref{comparison bundle}, the $\DIFF$ stable vertical tangent bundle constructed in \cref{sec:StabNormBundle} coincides with the one of \cite{ERW}. Likewise, the discussion after \cref{def euler} shows that the fibrewise Euler class constructed in \cref{sec:ConstEul} restricts to that of \cite{RW}. By \cite[Lemma 2.2 (iv)]{RW} the claim follows.
Now consider the case of topological bundles. By construction (see \cite[Proposition 4.2]{ERW}) it suffices to treat the case of a manifold base, in which case \cref{compb} implies that the $\TOP$ vertical tangent bundle of \cite{ERW} stabilises to the $\TOP$ stable vertical tangent bundle constructed in \cref{sec:StabNormBundle}, and \cite[Lemma 2.2 (ii) \& (iv)]{RW} shows that the Euler class of the $\TOP$ vertical tangent bundle agrees with the fibrewise Euler class constructed in \cref{sec:ConstEul}. (These results are stated for $\CAT = \DIFF$ in \cite{RW}, but their proofs apply for $\CAT = \TOP$ too.) 
\end{proof}

Finally, let us warn the reader that they should resist the temptation to think that the tautological classes in $H^*(\B\TOPtw{\mathstrut}^{\>\!+}(M);R)$ behave like their counterparts in $H^*(\B\TOP^+(M);R)$ or $H^*(\B\DIFF^+(M);R)$ as there is no reason for the homomorphism
\[\kappa_{-}(M) \colon H^*(\BSTOP \times K(\Z,d);R) \lra H^{*-d}(\B\TOPtw{\mathstrut}^{\>\!+}(M);R)\]
 to factor through $H^*(\BSTOP(d);R)$. Indeed, in the smooth case \cite[Proposition 3.1]{RW}  implies that
\[\kappa_{e^2}(W_g) \neq \kappa_{p_n}(W_g) \in H^{2n}(\B\DIFFtw{\mathstrut}^{\>\!+}(W_g); \Q)\]
 for $W_g = (S^n \times S^n)^{\#g}$ and large $g \in \mathbb N$ and in \cite[Theorem 3]{ERW} the authors construct an $8$-manifold $M$ with 
 \[0 \neq \kappa_{p_5}(M) \in H^{12}(\B\DIFFtw{\mathstrut}^{\>\!+}(M);\Q).\]
 While these results exclude a factorisation
\[\kappa \colon H^*(\mathrm{BSO};\Q) \longrightarrow H^*(\mathrm{BSO}(d);\Q) \longrightarrow H^*(\B\DIFFtw{\mathstrut}^{\>\!+}(M);\Q)\]
 in the case of smooth block bundles, they do not suffice to exclude the analogue for topological block bundles. Indeed, by work of Weiss \cite{Weiss}, neither $e^2 = p_n \in H^{4n}(\BSTOP(2n);\Q)$ nor $0 = p_m \in H^{4n}(\BSTOP(2n);\Q)$ for $m > n$ hold in general. In fact, it seems to be our lack of knowledge of $H^*(\BSTOP(n);\Q)$ that prevents us from disproving this factorisation. 

In a similar direction, our methods do not lift \emph{all} rational tautological classes of topological fibre bundles to topological block bundles, as the ring $H^*(\BSTOP(n);\Q)$ is not generated by Euler and Pontryagin classes. This may be seen as follows. The space $\frac{\mathrm{STop}}{\mathrm{SO}}$ is rationally contractible (by \cite[p.\ 200]{KS}) and the map $\frac{\mathrm{STop}(n)}{\mathrm{SO}(n)} \to \frac{\mathrm{STop}}{\mathrm{SO}}$ is $(n+2)$-connected (by \cite[p.\ 246]{KS}), so  $\frac{\mathrm{STop}(n)}{\mathrm{SO}(n)}$ is rationally $(n+1)$-connected. Combined with Morlet's identification
\[\B\DIFF_\partial(D^n) \simeq \Omega^n_0 \left(\frac{\mathrm{STop}(n)}{\mathrm{SO}(n)}\right)\]
(see \cite[p.\ 241]{KS}) and Farrell--Hsiang's calculation \cite{FH} of the rational homotopy groups of $\B\DIFF_\partial(D^n)$, we find that for $n$ odd and large enough
\[\pi_i\left(\frac{\mathrm{STop}(n)}{\mathrm{SO}(n)}\right) \otimes \mathbb Q =
\begin{cases}
0 & 0 < i < n+4\\
\Q & i=n+4.
\end{cases}\]
This implies that $H^{n+4}\left(\frac{\mathrm{STop}(n)}{\mathrm{SO}(n)}; \Q\right) \cong \Q$ is the lowest degree non-vanishing cohomology group. As $\B\mathrm{SO}(n)$ has no rational cohomology in odd degrees, in the Serre spectral sequence for the fibration
\[\frac{\mathrm{STop}(n)}{\mathrm{SO}(n)} \lto \B\mathrm{SO}(n) \to \BSTOP(n)\]
it follows that the transgression 
\[\Q \cong H^{n+4}\left(\frac{\mathrm{STop}(n)}{\mathrm{SO}(n)};\Q\right) \lra H^{n+5}(\BSTOP(n);\Q)\] 
is injective. By definition its image vanishes in $H^{n+5}(\B\mathrm{SO}(n); \Q)$, so it cannot be a polynomial in the Euler and Pontryagin classes.

\section{Block homeomorphisms of aspherical manifolds}\label{sec4}

In the previous sections we have established that the tautological classes of smooth manifold bundles extend to topological block bundles. As this paper aims to understand the tautological classes for aspherical manifolds, we will now discuss the homotopy type of the space $\B\wt{\TOP}(M)$ provided $M$ is aspherical. This depends on what are called the full Farrell--Jones conjectures, see \cref{sfj}. We will call a group satisfying them a \emph{Farrell--Jones group}.

\subsection{The block Borel conjecture}
Let $M$ be an aspherical manifold and recall from the introduction that $M$ is said to satisfy the \emph{block Borel conjecture} if the canonical map
\[\iota \colon \B\wt{\TOP}(M) \lto \B\G(M) \]
is a weak equivalence, and 
the \emph{{\pbc} (resp.\ with $R$-coefficients)} if the restriction
\[\iota_h \colon \B\wt{\TOP}_h(M) \lto \B\G_0(M) \]
is  a weak equivalence (resp.\ an $R$-homology equivalence). The block Borel conjecture implies the {\pbc}, by pulling back the universal cover of the target, which in turn implies the {\pbc}  with $R$-coefficients for any ring $R$.

The next proposition partly concerns surgery on 4-dimensional topological manifolds, so we recall some terminology from that theory: a group is said to be \emph{good} if it satisfies the $\pi_1$-Null Disk Lemma, see e.g.\ \cite[Introduction]{FT}. Groups that are known to be good include elementary amenable groups and groups of subexponential growth, see \cite{Freedman} and \cite{FT}.

\begin{Prop}\label{FJ vs BBC}
Let $M$ be an aspherical manifold whose fundamental group is a Farrell--Jones group. 
\begin{enumerate}[(i)]
\item If either the dimension of $M$ is at least 5, or the dimension of $M$ is 4 and its fundamental group is good, then the block Borel conjecture holds for $M$. 
\item If the dimension of $M$ is 4, then the {\pbc} holds for $M$.
\end{enumerate}
\end{Prop}

This proposition implies that the block Borel conjecture holds for a very large class of aspherical manifolds, see \cref{status FJ conjecture}.

\begin{proof}
Let us first sketch the argument, following \cite[Proposition 0.3]{BL2}, for part (i) in the case $\dim(M) \geq 5$. Denote by $\G_s(M) \subseteq \G(M)$ the space of \emph{simple} homotopy self equivalences of $M$, which is a collection of path components of $\G(M)$ containing both $\G_0(M)$ and the image of $\TOPtw(M)$. We denote by $\G(M)/\wt{\TOP}(M)$ the fibre of $\iota$, and by $\G_s(M)/\wt{\TOP}(M) \subset \G(M)/\wt{\TOP}(M)$ the evident collection of path components. 
From surgery theory for $k \geq 1$ one has isomorphisms
\[\pi_k(\G_s(M)/\wt{\TOP}(M)) \cong \s^\TOP_\partial(M \times \Delta^k)\]
to the higher structure sets of $M$ appearing in the surgery exact sequence
\[\cdots \lra  L^q_{d+{k+1}}(\mathbb Z[\pi_1(M)]) \lra \s^\TOP_\partial(M \times \Delta^k) \lra \mathcal N^\TOP_\partial(M \times \Delta^k) \stackrel{\sigma}{\lra} L^q_{d+k}(\mathbb Z[\pi_1(M)]) \lra \cdots\]
and one has an inclusion $\pi_0(\G_s(M)/\wt{\TOP}(M)) \subseteq \s^\TOP(M)$. By the work of Ranicki, $\sigma$ can be identified with the assembly map 
\[ \mathrm{L}^q\Z\langle 1 \rangle_{d+k}(M) \lto \mathrm{L}^q_{d+k}(\Z\pi_1(M)) \] 
in (based) quadratic L-theory, see \cite[Theorem 18.5]{Ranicki} where $\mathbb L.$ denotes our $\mathrm L^q\mathbb Z\langle 1 \rangle$. As explained in \cite{BL2} the Farrell--Jones conjecture in K- and L-theory imply that for an aspherical manifold
\[\mathrm L^q\Z_*(M) \lto \mathrm L^q_*(\Z\pi_1(M))\]
is an isomorphism (the K-theoretic Farrell--Jones conjecture for $\pi_1(M)$ is used to change from universally decorated to based L-theory as explained in the proof of \cite[Proposition 0.3 (i)]{BL2}). Since the map 
\[\mathrm{L}^q\Z\langle 1 \rangle_*(M) \lra \mathrm{L}^q\Z_*(M)\] 
is injective in degree $d$ and an isomorphism in higher degrees, $\s_\partial^\TOP(M \times \Delta^k)$ is trivial for all $k$. The Farrell--Jones conjecture also implies that $\G_s(M) = \G(M)$, as their difference is measured by the Whitehead torsion which takes values in the cokernel of the assembly map in K-theory. It follows that $\G(M)/\wt{\TOP}(M)$ is contractible and hence that $\iota$ is a homotopy equivalence. 

The only point where we used the dimension assumption on $M$ was to make sure that the surgery sequence is exact. By Freedman's results, this also holds for $4$-manifolds whose fundamental group is good, see \cite[Theorem 11.3A]{FQ}. Thus by the same reasoning we obtain part (i) of the proposition also for $4$-manifolds with good fundamental group, see also \cite[Section 11.5]{FQ}.

To prove (ii),
we apply the arguments above to $M \times \Delta^1$, to see that $\pi_k(\G(M)/\TOPtw(M))$ vanishes for $k \geq 1$. Therefore $\pi_i(\iota)$ is injective for $i=1$ and bijective for $i \geq 2$. As it is injective on $\pi_1$ it follows that $\TOPtw_h(M)$ is path-connected, so that passing to universal covers then proves that $\iota_h$ is a weak equivalence, and hence that $M$ satisfies the {\pbc}.
\end{proof}

\begin{Rmk}
For our purposes it will often suffice to know that $\iota_h \colon \B\TOPtw_h(M) \to \B\G_0(M)$  is an equivalence after inverting 2, or even rationally. To obtain this weaker statement one need not assume the full Farrell--Jones conjectures, a variant of the L-theoretic conjecture is enough, as we will explain in \cref{pbc2}.
\end{Rmk}

\begin{Rmk}\label{vanishing low dimensions}
We do not know the validity of the block Borel conjecture for aspherical manifolds of dimension less than 4. The vanishing results for tautological classes, however, are known in dimensions at most 3 anyway (at least rationally in case of dimension $3$). Indeed, in dimension 1 this is a straightforward calculation since $S^1$-bundles can always be given linear structures, in dimension 2, the orientable aspherical manifolds are exactly the surfaces $\Sigma_g$ with $g \geq 1$. For $g \geq 2$ the space $\DIFF_h(\Sigma_g)$ is contractible by a result of \cite{EE}, so there is nothing to prove. For the torus $T^2$ one uses that the maps 
\[\DIFF(T^2) \lto \TOP(T^2) \lto \G(T^2)\] 
are homotopy equivalences and that the left translation map $ T^2 \to \G(T^2)$ factors over $\DIFF(T^2)$ and is an equivalence onto $\DIFF_h(T^2)$. Then one uses the same calculation as in the case of principal $S^1$-bundles.
For 3-dimensional manifolds there is a general vanishing result due to Ebert, see \cite[Corollary 1.3]{Ebert}: He proves that rational, tautological classes vanish in the cohomology of $\B\DIFF^+(M)$ even for non-aspherical $M$. For more details, see the discussion in \cite[p.\ 10]{BFJ}. 
\end{Rmk}

\subsection{The Farrell--Jones conjectures}\label{sfj} 
Let us report now on the status of the Farrell--Jones conjectures to convince the reader that their assumption is not too restrictive.
We state the following result for the class $\calfj$ of Farrell--Jones groups, that is those groups which satisfy the full Farrell-Jones conjectures; this terminology refers to the following version, the details of which are explained in \cite[Section 11]{Lueck}: A group $G$ is contained in $\calfj$ if for every finite group $F$ and every additive $G \wr F$-category $\mathcal A$ (with involution in the L-theoretic case) both maps
\[\mathrm K\mathcal A^{G \wr F}_n(E_\VC G \wr F) \lto \mathrm K_n(\mathcal A;G\wr F) \quad \text{ and } \quad \mathrm L\mathcal A^{G \wr F}_n(E_\VC G \wr F) \lto \mathrm L_n(\mathcal A;G\wr F)\]
are isomorphisms for all $n \in \mathbb Z$. Here, K denotes non-connective K-theory and L universally decorated, that is $\langle -\infty\rangle$, L-theory and $E_\VC$ denotes the classifying space for the family of virtually cyclic subgroups.

This version of the Farrell--Jones conjecture contains the more classical form saying that for any ring $R$ and a discrete group $G$, the assembly maps
\[ \mathrm K R^G_*(E_\VC G) \lto \mathrm K_*(RG) \quad \text{ and } \quad \mathrm L R^G_*(E_\VC G) \lto \mathrm L_*(RG)\]
are isomorphisms. For a torsion-free group these two conjectures imply that
\[ \mathrm{K}\Z_*(\B G) \lto \mathrm{K}_*(\Z G) \quad \text{ and } \quad \mathrm{L}\Z_*(\B G) \lto \mathrm{L}_*(\Z G)\] 
are equivalences (even for decoration $\langle 2 \rangle$, i.e.\ based, L-theory as needed in the proof of \cref{FJ vs BBC}): Any torsion-free virtually cyclic group is in fact infinite cyclic (this follows easily from \cite[Lemma 2.5]{FJ}) so $E_\VC G = E_{\mathrm{cyc}} G$. But the relative assembly maps (with universal decorations in L-theory)
\[ \mathrm{K}\Z_*(\B G) \lto \mathrm{K}\Z^G_*(E_{\mathrm{cyc}} G) \quad \text{ and } \quad \mathrm{L}\Z_*(\B G) \lto \mathrm{L}\Z^G_*(E_{\mathrm{cyc}} G)\] 
are isomorphisms for every group, see \cite[Proposition 2.10 (ii)]{LR}. Finally, the K-theoretic conjecture allows one change to the desired decorations. 

\begin{Thm}\label{status FJ conjecture}
The class $\calfj$ has the following properties
\begin{enumerate}
\item[(i)] \label{the:status_of_the_Full_Farrell-Jones_Conjecture:Classes_of_groups:hyperbolic_groups}
It contains hyperbolic groups and finite dimensional $\CAT(0)$-groups;
\item[(ii)] \label{the:status_of_the_Full_Farrell-Jones_Conjecture:Classes_of_groups:solvable}
it contains virtually solvable groups;
\item[(iii)] \label{the:status_of_the_Full_Farrell-Jones_Conjecture:Classes_of_groups:lattices}
it contains (not necessarily cocompact) lattices in almost connected Lie groups;
\item[(iv)] \label{the:status_of_the_Full_Farrell-Jones_Conjecture:Classes_of_groups:S-arithmetic}
it contains $S$-arithmetic groups;
\item[(v)] it is closed under passing to subgroups;
\item[(vi)] it is closed under taking finite products, free products, and directed colimits;
\item[(vii)] it is almost closed under extensions, more precisely, let $ 1 \to K \to G \to Q \to 1$ be an extension of groups. Suppose that for any cyclic subgroup 
$C \subseteq Q$ the group $p^{-1}(C)$ belongs to $\calfj$ 
and that the group $Q$ belongs to $\calfj$.
Then $G$ belongs to $\calfj$;
\item[(viii)] if $H$ is a finite index subgroup of $G$, and $H$ is in $\calfj$, then also $G$ is in $\calfj$.
\end{enumerate}
\end{Thm}
For a more complete and detailed status of the Farrell--Jones conjectures we refer the reader to \cite{RV, BL2, BFL, KLR, R} and the references therein.

If one is willing to neglect 2-torsion, then the L-theoretic Farrell--Jones conjecture has further useful properties. For this we need to recall the following version known as the \emph{fibered} Farrell--Jones conjecture, see \cite[Section 2.1]{BLR2}. For a fixed group $G$ its L-theoretic version after inverting 2 states that for any group homomorphism $\phi \colon H \to G$ the assembly map
\[ \mathrm L R^H_*(E_{\phi^*(\VC)}H)[\tfrac{1}{2}] \lto \mathrm L_*(RH)[\tfrac{1}{2}]\]
is an isomorphism; here, for any family $\mathcal{F}$ of subgroups of $G$ and a homomorphism $\phi \colon H \to G$, we denote by $\phi^*(\mathcal F)$ the family of subgroups $K \subseteq H$, such that $\phi(K) \in \mathcal F$. After inverting 2, L-theory spectra with different decorations become equivalent, see \cite[Remark 1.22]{LR}, therefore we do not need to consider the K-theoretic analogue for this. We will denote by $\LFJ$ the class of groups that satisfy this conjecture. Since one can choose the homomorphism $\id \colon G \to G$, it contains the classical version of the assembly maps with $2$ inverted. The class $\LFJ$ has better closure properties than $\calfj$, as described in the following proposition.

\begin{Prop}\label{LFJ}\mbox{}
\begin{enumerate}
\item[(i)] Farrell--Jones groups lie in $\LFJ$,
\item[(ii)] elementary amenable groups are contained in $\LFJ$,
\item[(iii)] it is closed under passage to subgroups, directed colimits, and amalgamated products, and
\item[(iv)] it is almost closed under extensions, more precisely, let $ 1 \to K \to G \to Q \to 1$ be an extension of groups. Suppose that for any finite subgroup
$C \subseteq Q$ the group $p^{-1}(C)$ belongs to $\LFJ$ and that the group $Q$ belongs to $\LFJ$.
Then $G$ belongs to $\LFJ$.
\end{enumerate}
\end{Prop}

Note that, in particular, $\LFJ$ is closed under extensions of torsion-free groups, which will be the salient feature of (iv) later.

To show \cref{LFJ} we need another characterisation of $\LFJ$. To state it let us denote by $\LFJfin$ the class of groups satisfying the variant of the fibered conjecture which takes into account the family of finite subgroups rather than the family of virtually cyclic subgroups.

\begin{Lemma}\label{after inverting 2 in fibered case fin equals vc}
We have that $\LFJ = \LFJfin$.
\end{Lemma}

\begin{proof}
By \cite[Theorem 2.4]{BLR2} this follows if we can show that every virtually cyclic group $V$ is contained in $\LFJfin$. From \cite[Lemma 2.5]{FJ} it follows that $V$ sits inside a short exact sequence
\[ 1 \lto F \lto V \lto S \lto 1 \]
in which $F$ is a finite group and $S$ is 
either infinite cyclic or infinite dihedral, i.e.\ isomorphic to $D_\infty \cong \Z/2 \ast \Z/2$.
By \cite[Lemma 2.9]{BLR2} and the fact that finite groups are clearly contained in $\LFJfin$, it hence suffices to prove that $\Z$ and $D_\infty$ are contained in $\LFJfin$.

To do so let $\phi \colon K \to S$ be a homomorphism with $S$ either $\mathbb Z$ or $D_\infty$. We need to show that $K$ satisfies the L-theoretic Farrell--Jones conjecture after inverting 2 with respect to the family 
$\phi^*(\fin)$. Factor $\phi$ as 
\[ K \stackrel{\psi}\lto \mathrm{Im}(\phi) \stackrel{i}{\lto} S\]
and observe that $\phi^*(\fin) = \psi^*(\fin)$. For the image of $\phi$ -- as for every subgroup of $D_\infty$ -- there are three possibilities: It is either finite, infinite cyclic or infinite dihedral. If the image of $\phi$ is finite, $\phi^*(\fin)$ is the family of all subgroups and thus $K$ clearly satisfies this isomorphism conjecture and otherwise we are reduced to considering a \emph{surjection} $\phi\colon K \rightarrow S$.
Notice that the space $E_\fin S$ acquires a $K$-action through the homomorphism $\phi$ and that $E_\fin S$ is a model for $E_{\phi^*(\fin)}K$ when $\phi$ is surjective. Let $K_0$ denote the kernel of $\phi$.

We will now proceed by studying the two cases separately. Let us first assume that $S = \mathbb Z$. As $\mathbb Z$ has no finite subgroups a model for the space $E_{\phi^*(\fin)} K$ is given by $E\Z$ and as a $K$-CW complex is given by the pushout
\[\xymatrix{ K/K_0 \times S^0 \ar[r] \ar[d] & K/K_0 \ar[d] \\ K/K_0 \times D^1 \ar[r] & E_{\phi^*(\fin)}K }\]
Consider then the diagram
\[\xymatrix{\dots \ar[r] & \mathrm{L}R^K_*(K/K_0) \ar[r] \ar[d] & \mathrm{L}R_*^K(K/K_0) \ar[r] \ar[d] & \mathrm{L}R^K_*(E_{\phi^*(\fin)}K) \ar[r] \ar[d] & \mathrm{L}R^K_{*-1}(K/K_0) \ar[r] \ar[d] & \dots \\
		\dots \ar[r] & \mathrm{L}_*(RK_0)[\tfrac{1}{2}] \ar[r] & \mathrm{L}_*(RK_0)[\tfrac{1}{2}] \ar[r] & \mathrm{L}_*(RK)[\tfrac{1}{2}] \ar[r] & \mathrm{L}_{*-1}(RK_0)[\tfrac{1}{2}] \ar[r] & \dots }\]
where the upper horizontal sequence is the exact sequence induced by applying $\mathrm L R^K_*(-)$ to the above pushout, the lower horizontal sequence is the exact sequence of \cite[page 413]{Ranicki3} and the vertical arrows are given by the assembly map. By definition of the equivariant homology theory $\mathrm{L}R^K_*(-)$, the assembly maps involving only $K/K_0$ become isomorphisms after inverting 2 in the domain. We deduce that the map
\[ \mathrm{L}R^K_*(E_{\phi^*(\fin)}K)[\tfrac{1}{2}] \lto \mathrm{L}_*(RK)[\tfrac{1}{2}] \]
is an isomorphism from the 5-lemma.

To address the case $S = D_\infty$ note that a model for $E_\fin(D_\infty)$ is given by the pushout of $D_\infty$-CW complexes
\[\xymatrix{ D_\infty \times S^0 \ar[r] \ar[d] & D_\infty/C_1 \amalg D_\infty/C_2 \ar[d] \\ D_\infty \times D^1 \ar[r] & E_\fin(D_\infty)}\]
where $C_1 = \Z/2 \ast \{ e \}$ and $C_2 = \{ e\} \ast \Z/2$ are the canonical subgroups. 
As explained above, it follows that a model for $E_{\phi^*(\fin)} K$ is given by the pushout of $K$-CW complexes
\[ \xymatrix{K/K_0 \times S^0 \ar[r] \ar[d] & K/K_1 \amalg K/K_2 \ar[d] \\ K/K_0 \times D^1 \ar[r] & E_{\phi^*(\fin)}K }\]
where $K_0 = \ker(\phi)$ and $K_i = \phi^{-1}(C_i)$ for $i = 1,2$. 
We consider the diagram
\[ \xymatrix@C=.5cm{\cdots \mathrm{L}R^K_*(K/K_0) \ar[r] \ar[d] & \mathrm{L}R^K(K/K_1) \oplus \mathrm{L}R^K_*(K/K_2) \ar[r] \ar[d] & \mathrm{L}R^K_*(E_{\phi^*(\fin)}K) \ar[r] \ar[d] & \mathrm{L}R^K_{*-1}(K/K_0) \ar[d] \cdots \\
		\cdots \mathrm{L}_*(RK_0)[\tfrac{1}{2}] \ar[r] & \mathrm{L}_*(RK_1)[\tfrac{1}{2}] \oplus \mathrm{L}_*(RK_2)[\tfrac{1}{2}] \ar[r] & \mathrm{L}_*(RK)[\tfrac{1}{2}] \ar[r] & \mathrm{L}_{*-1}(RK_0)[\tfrac{1}{2}]  \cdots }\]
where again the upper horizontal sequence is the exact sequence induced by applying $\mathrm{L}R^K_*(-)$ to the above pushout, whereas the lower horizontal sequence is the exact sequence of \cite[Corollary 6]{Cappell}. The vertical maps are again the assembly maps and thus isomorphisms at the terms involving only homogeneous spaces in the source. The $5$-lemma again finishes the proof.
\end{proof}

\begin{proof}[Proof of \cref{LFJ}]
Statement (i) is proven in \cite[Corollary 4.3]{BR}. In \cite[Lemma 2.12]{BLR2} it is shown that elementary amenable groups are contained in $\LFJfin$, even without inverting 2. By \cref{after inverting 2 in fibered case fin equals vc} we deduce (ii). Part (iii) is \cite[Lemma 2.5 \& Theorem 2.7]{BLR2} for subgroups and directed colimits. The closure property under amalgamated product follows from \cite[Corollary 6]{Cappell} using the fact that for a surjective group homomorphism $\phi \colon K \to G$ a decomposition $G = G_1 \ast_{G_0} G_2$ induces a decomposition $K = \phi^{-1}(G_1) \ast_{\phi^{-1}(G_0)} \phi^{-1}(G_2)$, similar to the argument in \cref{after inverting 2 in fibered case fin equals vc}. 
It hence remains to prove part (iv). This follows immediately from \cite[Lemma 2.9]{BLR2} using \cref{after inverting 2 in fibered case fin equals vc}.
\end{proof}

\begin{Prop}\label{pbc2}
Let $M$ be an aspherical manifold of dimension at least 4. If $\pi_1(M)$ is contained in $\LFJ$, then the {\pbc} with $\Z[\tfrac{1}{2}]$-coefficients holds for $M$. 
\end{Prop}

\begin{proof}
As explained previously, the difference between the decorations in L-theory disappears after inverting $2$ (by \cite[Remark 1.22]{LR}). From the identification of the surgery obstruction with the assembly map, as expounded in the proof of \cref{FJ vs BBC}, one therefore obtains that the groups $\s^\TOP_\partial(M \times \Delta^k)[\tfrac 1 2]$ vanish. For $k > 0$ this implies that $\pi_k(\G_s(M)/\TOPtw(M))[\tfrac{1}{2}] =0$ (the fundamental group is abelian by \cite[Theorem 18.5]{Ranicki}); note that no statement about the components can be deduced even if $M$ is of dimension greater that $4$. We conclude that the map $\B\TOPtw(M) \to \B\G_s(M)$ induces an isomorphism on $\pi_i(-)[\tfrac{1}{2}]$ for $i > 1$, and the kernel of the map on $\pi_1$ is abelian and annihilated by inverting 2. It follows that the homotopy fibre of $\iota_h \colon \B\TOPtw_h(M) \to \B\G_0(M)$ has abelian fundamental group and all homotopy groups annihilated by inverting 2. By considering its Postnikov tower we then see that it has the $\Z[\tfrac{1}{2}]$-homology of a point, and hence that $\iota_h$ is a $\Z[\tfrac{1}{2}]$-homology equivalence.
\end{proof}

\subsection{Block homeomorphisms of aspherical manifolds}\label{sec:BlockHomeos} 

The block Borel conjecture implies a strong computational result, namely a full understanding of the homotopy type of the space $\B\wt{\TOP}(M)$, as we will see in \cref{homotopy type of Block homeo}. This result, together with the fact that the tautological classes are defined in $H^*(\B\wt{\TOP}{\mathstrut}^{\>\!+}(M))$, as discussed in \cref{def kappa}, is key to our approach to understanding the tautological classes for aspherical manifolds.

It is straightforward to show (see e.g.\ \cite[section III]{Gottlieb}) that if $\Gamma$ is a group then there is a canonical fibre sequence
\begin{equation*}
 \B^2 \C(\Gamma) \lto \B\G(\B\Gamma) \lto \B\Out(\Gamma)
\end{equation*}
where $\C(\Gamma)$ denotes the centre of $\Gamma$ and $\Out(\Gamma)$ denotes the group of outer automorphisms of $\Gamma$. 
From this point onwards, we will let $\Gamma$ be the fundamental group of an aspherical manifold $M$. We can draw the following corollary.

\begin{Prop}\label{homotopy type of Block homeo}
Let $M$ be an aspherical manifold satisfying the block Borel conjecture. Then there is a fibre sequence
\[\B^2 \C(\Gamma) \lto \B \wt{\TOP}(M) \lto \B\Out(\Gamma).\]
\end{Prop}

Recall that we call an aspherical manifold $M$ \emph{centreless} if $\C(\Gamma) = 0$. We immediately obtain:

\begin{Cor}\label{centreless}
Let $M$ be a closed, connected, oriented, centreless, aspherical, manifold of dimension $d$, that satisfies the {\pbc} with $R$-coefficients. Then 
\[0 = \kappa_c(M) \in H^{k-d}(\BTOPtw_h(M);R)\]
for all $c \in H^k(\BSTOP \times K(\mathbb Z,d);R)$ with $k \neq d$.
\end{Cor}

One large class of examples of centreless aspherical manifolds is given by those admitting a metric of negative sectional curvature. Just as in \cite{BFJ} one can strengthen our results for such manifolds. To this end recall, that the fundamental group of a negatively curved manifold is \emph{hyperbolic}. 

\begin{Cor}\label{cohomology of block top}
Let $M$ be a closed, oriented, aspherical manifold of dimension $d \geq 4$ with hyperbolic fundamental group.
Then 
\[ H^*\big(\BTOPtw{\mathstrut}^{\>\!+}(M);\Q\big) = \Q.\]
In particular, 
\[0 = \kappa_c(M) \in H^{k-d}(\BTOPtw{\mathstrut}^{\>\!+}(M); \mathbb Q)\]
for all $c \in H^k(\BSTOP \times K(\mathbb Z,d);\mathbb Q)$ with $k \neq d$.
\end{Cor}

\begin{proof}
Let us collect two relevant features of hyperbolic groups. 

Firstly, a torsion-free hyperbolic group different from $\Z$ has trivial centre. This is well-known, but as we had difficulties finding it in the literature we give the short proof. Suppose that the centre of such a group $\Gamma$ is non-trivial, and let $x \in \C(\Gamma)$ be a non-trivial element. By \cite[Corollary 3.10]{BH} we have that $\langle x \rangle$ has finite index in its centraliser, but since $x$ is central its centraliser is the whole of $\Gamma$, and so $\Gamma$ is virtually infinite cyclic. As mentioned earlier, it follows directly from \cite[Lemma 2.5]{FJ} that a torsion-free virtually cyclic group is in fact infinite cyclic.

Secondly, Gromov has shown \cite[Theorem 5.4.A]{Gromov} that a hyperbolic group which is the fundamental group of an aspherical manifold of dimension at least 3 has finite outer automorphism group.

We now begin the proof of this corollary. As $\pi_1(M)$ has trivial centre, the map $\G(M) \to \Out(\pi_1(M))$ is a homotopy equivalence, and by Gromov's theorem the latter is a finite group. By \cref{status FJ conjecture} (i) hyperbolic groups satisfy the Farrell--Jones conjecture, and hence by \cref{FJ vs BBC} the manifold $M$ satisfies the block Borel conjecture if $d \geq 5$, and the identity block Borel conjecture if $d=4$. But in the latter case the proof of \cref{FJ vs BBC} also shows that $\iota\colon \B\TOPtw(M) \to \B\G(M)$ is injective on $\pi_1$. In either case it follows that $\TOPtw(M)$ is homotopy equivalent to a finite group, so its classifying space has trivial rational cohomology.
\end{proof}

\begin{Rmk}\label{moge}
The assumption on the fundamental group of $M$ can of course be relaxed: the conclusion of the corollary holds whenever $M$ is a centreless manifold whose fundamental group is a Farrell--Jones group that has rationally acyclic, e.g.\ finite, outer automorphism group. In a similar vein, one can ask for $\Out(\pi_1(M))$ to have finite rational cohomological dimension. In this case one still obtains that the tautological classes of non-zero degree are nilpotent in $H^*\big(\BTOPtw{\mathstrut}^{\>\!+}(M); \mathbb Q\big)$, a claim we will make again later.
\end{Rmk}

\begin{Rmk}
The stronger statement explained in \cref{moge} recovers \cite[Theorem F]{BFJ} where such a vanishing is proven for smooth bundles with fibre a non-positively curved centreless manifold whose fundamental group has finite outer automorphism group: fundamental groups of non-positively curved manifolds are $\CAT(0)$ and thus Farrell--Jones groups. Notice that the contents of \cref{moge} for smooth bundles are also implied by \cite[Corollary G.1]{BFJ}.
\end{Rmk}

The opposite extreme to centreless aspherical manifolds are those with abelian fundamental groups, that is, tori $T^d$. In this case there are also special features which allow us to prove vanishing of tautological classes on $\BTOPtw{\mathstrut}^{\>\!+}(T^d)$ and not just $\BTOPtw{\mathstrut}_0(T^d)$. The following specialises to \cite[Corollary D.1]{BFJ} in the case of $\Q$-coefficients.

\begin{Cor}\label{npt}
For $d \geq 4$ we have
\[0 = \kappa_c(T^d) \in H^{k-d}(\BTOPtw{\mathstrut}^{\>\!+}(T^d); R)\]
for all $c \in H^k(\BSTOP \times K(\mathbb Z,d);R)$.
\end{Cor}

\begin{proof}
First we note that finitely generated free abelian groups are good in the sense of Freedman, see \cite[after Disk Theorem 1]{Freedman} and are Farrell--Jones groups by \cref{status FJ conjecture}. 
Thus by \cref{FJ vs BBC} the block Borel conjecture is valid for $T^d$. Hence by \cref{homotopy type of Block homeo} we have a fibre sequence 
\[\B^2 \mathbb Z^d \lto \BTOPtw{\mathstrut}^{\>\!+}(T^d) \lto \B\mathrm{SL}_d(\mathbb Z).\]
A simple calculation in the long exact sequence of homotopy groups associated to a fibration shows that the total space of the universal, oriented $T^d$-block bundle $\pi\colon \widetilde E^+(T^d) \rightarrow \BTOPtw{\mathstrut}^{\>\!+}(T^d)$ is given by $\B\mathrm{SL}_d(\mathbb Z)$. Moreover, the composite
\[\B\mathrm{SL}_d(\mathbb Z) \stackrel{\pi}\longrightarrow \BTOPtw{\mathstrut}^{\>\!+}(\B\Gamma) \longrightarrow \B\mathrm{SL}_d(\mathbb Z)\]
is a homotopy equivalence. It follows that the induced map 
\[\pi^*\colon H^*(\BTOPtw{\mathstrut}^{\>\!+}(T^d); R) \longrightarrow H^*(\widetilde E^+(T^d);R)\] 
is (split) surjective for all coefficients. But for all $x \in H^*(\BTOPtw{\mathstrut}^{\>\!+}(T^d); R)$ we have
\[\pi_!(\pi^*(x)) = x \smile \pi_!(1) = 0.\]
Since all tautological classes lie in the image of $\pi_!$ the claim follows.
\end{proof}

\begin{Rmk} 
For smooth fibre bundles and rational coefficients, the same vanishing result is also true in dimensions less than $4$:
When $d = 1,2$ the universal smooth $T^d$-bundle is given by $\mathrm{BSL_d}(\mathbb Z) \rightarrow \B(T^d \rtimes \mathrm{SL_d}(\mathbb Z))$ and the same proof applies. For $T^3$ the rational result follows from Ebert's work \cite{Ebert}, as mentioned earlier in \cref{vanishing low dimensions}.
\end{Rmk}

\section{Vanishing criteria for tautological classes of aspherical manifolds}\label{sec5}

In this section we shall introduce Burghelea's conjecture and mostly restrict to rational coefficients throughout. We will prove our main theorem from the introduction and discuss its integral refinement at the end of the section. For a group $\Gamma$, we denote by $\C(\Gamma)$ its centre, and for an element $g \in \Gamma$, we denote by $\C_\Gamma(g)$ its centraliser in $\Gamma$. Furthermore, $\cd_\Q$ denotes the rational cohomological dimension and $\cd_\Q^\tr$ denotes the rational cohomological dimension with trivial coefficients. We will say that a group $\Gamma$ of type $F$ is called an \emph{oriented rational Poincar\'e duality group} \cite[VIII, Section 10]{Brown} of formal dimension $d$, if $H^*(\Gamma;\Q\Gamma)$ is concentrated in degree $d$, and there is isomorphic to $\Q$ with trivial $\Gamma$ action. In this case, cap product with a generator of this group yields an isomorphism
\[H^k(\Gamma ; M) \lra H_{d-k}(\Gamma ; M)\]
for any $\Q\Gamma$-module $M$.

\subsection{Relating tautological classes to Burghelea's conjecture}\label{xyz}
The basic ingredient into our study of tautological classes for not necessarily centreless aspherical manifolds is the following lemma. 

\begin{Lemma}\label{blockfib}
For an aspherical manifold $M$ with fundamental group $\Gamma$, the universal $M$-fibration over $\B\G_0(M)$ is given by
\[\B\Gamma \lra \B(\Gamma/\C(\Gamma)) \stackrel{\pi}\lra \B^2\C(\Gamma),\]
where $\pi$ classifies the central extension
\[1 \lto \C(\Gamma) \lra \Gamma \lra \Gamma/\C(\Gamma) \lto 1.\]
\end{Lemma}
\begin{proof}
The map $\B^2 \C(\Gamma) \to \B\G(\B\Gamma)$ from the beginning of \cref{sec:BlockHomeos} classifies a $\B\Gamma$-fibration $\B\Gamma \to E \to \B^2 \C(\Gamma)$. The connecting map $\partial : \C(\Gamma) = \pi_2(\B^2 \C(\Gamma)) \to \Gamma = \pi_1(\B\Gamma)$ in the long exact sequence on homotopy groups for this fibration is the inclusion map, and so $E \simeq \B(\Gamma/C(\Gamma))$ as required.
\end{proof}

If $M$ satisfies the {\pbc} with $R$-coefficients then there is a homotopy cartesian square
\[\xymatrix{
\widetilde E_h(M) \ar[d] \ar[r] & \B(\Gamma/C(\Gamma)) \ar[d]^-\pi\\
\BTOPtw_h(M) \ar[r]^{\iota_h}& \B^2 C(\Gamma)
}\]
where the left-hand map is the universal $M$-block bundle over $\BTOPtw_h(M)$ and the lower map is an $R$-homology equivalence. As $\B^2 C(\Gamma)$ is simply connected it follows that the top map is also an $R$-homology equivalence, so in order to show the vanishing of tautological classes it will therefore suffice to show that
\[\pi_!\colon H^k(\B(\Gamma/\C(\Gamma));R) \lra H^{k-d}(\B^2 \C(\Gamma);R)\]
is the zero map, which is precisely what we will do in this section for $R=\Q$.

To this end we will introduce various finiteness conditions and already want to offer following diagram to sum up the various implications among them:
\[\xymatrix{ &&\CBC \ar@{<=>}[rr]^{F_\mathbb Q} &\ & \IKP \ar@{<=>}[rr]^-{\text{Poincar\'e}} & \ & \IVP \ar@{=>}[d]_{\C \neq 0} & \text{$\kappa$-classes vanish}\\
            \BC \ar@{=>}[rru]&&\CP \ar@{=>}[rr]_{\C \neq 0} \ar@{=>}[rru] &\ & \KC \ar@{=>}[rr]_-{\text{Poincar\'e}} &\ & \GC \ar@{=>}[ur]_{\text{\;\;\; block Borel}} &  }\]
Here the fundamental group of an aspherical manifold lies in $\CBC$ if and only if it satisfies the central part of Burghelea's conjecture as stated in the introduction, thereby establishing our main theorem. The other terms are introduced throughout the section. To get started, let us axiomatise the conclusion we want to obtain.

\begin{Def}
Let $\GC$ (vanishing property) denote the class of oriented rational Poincar\'e duality groups of some dimension $d$ for which the Gysin map 
\[\pi_!\colon H^*(\B(\Gamma/\C(\Gamma));\Q) \lra H^{*-d}(\B^2\C(\Gamma);\Q)\] 
vanishes. Similarly, let $\IVP$ (individual vanishing property) consist of those oriented rational Poincar\'e duality groups for which 
\[\rho_!\colon H^*(\B(\Gamma/\langle g\rangle);\Q) \lra H^{*-d}(\B^2\Z;\Q)\] 
vanishes for each central $g \in \Gamma$ of infinite order individually; here $\rho \colon \B(\Gamma/\langle g\rangle) \to \B^2 \Z$ classifies the extension given by $g$.
\end{Def}

Assuming the {\pbc} with $\Q$-coefficients, $\pi_1(M)$ lying in  $\GC$ implies the vanishing of \emph{all} tautological classes in $H^*(\BTOPtw_h(M); \mathbb Q)$ for any oriented, aspherical manifold as we have explained above. We mainly introduce the class $\IVP$ to connect our conjecture from the introduction to Burghelea's, see below. To start this off we have the following.

\begin{Prop}\label{Oscar's trick}
A group $\Gamma$ in $\IVP$ lies in $\GC$ if and only if $\C(\Gamma) \otimes \Q \neq 0$.
\end{Prop}

\begin{proof}
If $\C(\Gamma) \otimes \Q = 0$, the rational Gysin map is isomorphic to the Gysin map for the trivial fibration $\B\Gamma \to *$
which by Poincar\'e duality for $\Gamma$ is non-zero in degree $d$ and thus $\Gamma$ does not lie in $\GC$.

To prove the converse observe that $H^*(\B^2 \C(\Gamma); \Q)$ is the symmetric algebra on the finite dimensional graded vector space $\Hom(\C(\Gamma), \mathbb Q)[2]$: Its dimension equals the rational cohomological dimension of $\C(\Gamma)$, which is bounded by that of its ambient group $\Gamma$. Now suppose that $x \in H^k(\B(\Gamma/\C(\Gamma));\Q)$ has $\pi_!(x) \neq 0 \in H^{k-d}(\B^2C(\Gamma);\Q)$. Then we claim that there is an embedding $i\colon \Z \to \C(\Gamma)$ such that $(\B^2i)^*\pi_!(x) \neq 0 \in H^{k-d}(\B^2\Z;\Q)$. Assuming this claim for the moment, we consider the diagram
\[\xymatrix{\B(\Gamma / \Z) \ar[r]^-{\rho} \ar[d]_{\B(\Gamma/i)} & \B^2 \mathbb Z \ar[d]^{\B^2 i}\\
            \B(\Gamma / \C(\Gamma)) \ar[r]^-\pi                    & \B^2 \C(\Gamma)}\]
which exhibits $\B(\Gamma / \Z)$ as a homotopy pullback. Therefore the diagram
\[\xymatrix{H^*(\B(\Gamma / \C(\Gamma));\Q) \ar[r]^-{\pi_!}\ar[d]_{{\B(\Gamma/i)}^*} & H^{*-d}(\B^2 \C(\Gamma);\Q) \ar[d]^{(\B^2 i)^*}\\
            H^*(\B(\Gamma / \Z);\Q) \ar[r]^-{\rho_!}                         & H^{*-d}(\B^2 \Z;\Q)}\]
commutes, which implies $\rho_!(\B(\Gamma/i)^*(x)) = (\B^2i)^*\pi_!(x) \neq 0$,  a contradiction as $\rho_!$ is zero by assumption.

To prove the claim we consider a general non-zero element $y \neq 0 \in H^{2n}(\B^2C(\Gamma);\Q) = \Sym^n(\Hom(\C(\Gamma), \mathbb Q))$. Such a $y$ is a non-zero polynomial function on $\C(\Gamma) \otimes \Q$ so since $\C(\Gamma)\otimes \Q$ is non-zero, there must be some non-zero element $v \in \C(\Gamma)\otimes \Q$ on which $y$ does not vanish: as $y$ is homogeneous it does not vanish on the entire line spanned by $v$ except at the origin. Such a line contains the non-trivial image of an element $w \in \C(\Gamma)$, and the homomorphism $i\colon \mathbb Z \to \C(\Gamma)$ defined by $w$ has the desired properties, since $i^*$ precisely corresponds to restriction of functions to the line spanned by $v$. 
\end{proof}

Let us now recall Burghelea's conjecture in full, see \cite{Burghelea}. We first state its conclusion in an axiomatic way, since the known cases go beyond Burghelea's original conjecture.

\begin{Def}
Let $\BC$ (Burghelea property) denote the class of groups $\Gamma$ that satisfy the following: For any element $g \in \Gamma$ of infinite order we have that the limit of
\[\xymatrix{\dots \ar[r] & H_{*+4}(\C_\Gamma(g)/\langle g \rangle;\Q) \ar[r]^-{-\frown e}  & H_{*+2}(\C_\Gamma(g)/\langle g \rangle;\Q) \ar[r]^-{-\frown e}& H_*(\C_\Gamma(g)/\langle g \rangle;\Q) }\]
vanishes, where $e \in H^2(\C_\Gamma(g)/\langle g \rangle;\Q)$ is the Euler class of the central extension
\[1 \lra \mathbb Z \stackrel{g}{\lra} \C_\Gamma(g) \lra \C_\Gamma(g)/\langle g \rangle \lra 1.\]
Let furthermore $\CBC$ (central Burghelea property) denote the class of groups where the same conclusion need only hold for central elements. 
\end{Def}

\begin{conj}[Burghelea]
Any group of type $F$ is in $\BC$.
\end{conj}

Recall that a group is said to be of type $F$ if there exists a model of its classifying space which is a finite complex, in particular the fundamental group of any aspherical manifold is of type $F$. Combining \cref{blubb} and \cref{Markus' trick} below shows that for group of type $F$ being in $\CBC$ is indeed equivalent to the central part of Burghelea's conjecture as stated in the introduction, justifying the name.

We will review known results about Burghelea's conjecture in the final chapter. For now let it suffice to say that it is known to be true for several classes of groups and that while some groups are known to lie outside of $\BC$ none of them are of type $F$. In order to connect Burghelea's conjecture to ours we need yet another definition.

\begin{Def}
Let $\KC$ (kernel property) denote the class of groups $\Gamma$, such that $\cd_\mathbb Q(\Gamma) < \infty$ and the map 
\[\pi^* \colon H^*(\B^2\C(\Gamma);\Q) \lra H^*(\B(\Gamma/\C(\Gamma));\Q)\] 
is not injective.

Similarly, let $\IKP$ (individual kernel property) denote those $\Gamma$ with $\cd_\Q(\Gamma) < \infty$, such that for each central $g \in \Gamma$ of infinite order the induced map 
\[\rho^* \colon H^*(\B^2\Z;\Q) \lra H^*(\B(\Gamma/\mathcal \langle g\rangle);\Q)\] 
is not injective. 
\end{Def}

Equivalently, one can describe groups in $\IKP$ by requiring the rational Euler class $e \in H^2(\B(\Gamma/\mathcal \langle g\rangle);\Q)$ of the central extensions $1 \to \Z \overset{g}\to \Gamma \to \Gamma/\langle g\rangle \to 1$ to be nilpotent.
\
\begin{Rmk}\label{blubb}
It follows from the Gysin sequence that multiplication with the Euler class 
\[e \smile - \colon H^*(\Gamma/\langle g \rangle;\Q) \lra H^{*+2}(\Gamma/\langle g \rangle;\Q)\]
 is an isomorphism in degrees greater than the cohomological dimension of $\Gamma$, so if $\cd_\mathbb Q(\Gamma) < \infty$ and the Euler class is nilpotent then we must have $\cd_\Q^\tr(\Gamma/\langle g \rangle) < \infty$. This means that $\Gamma \in \IKP$ if and only if $\cd_\Q(\Gamma)$ and $\cd_\Q^\tr(\Gamma/\langle g \rangle)$ are finite for each central $g \in \Gamma$ of infinite order. In particular, it shows that if $\Gamma \in \IKP$ then $\Gamma \in \CBC$.
\end{Rmk}

\begin{Prop}\label{Markus' trick}
Let $\Gamma \in \CBC$ be a group with $\cd_\mathbb Q(\Gamma) < \infty$, whose rational homology is of finite type. Then $\Gamma \in \IKP$.
\end{Prop}


\begin{proof}
Let $g \in \Gamma$ be central of infinite order. As the rational homology of $\Gamma$ is of finite type, so is that of $\Gamma/ \langle g \rangle$ by the Serre spectral sequence for $\B\Gamma \rightarrow \B\Gamma/\langle g \rangle \rightarrow \B^2\mathbb Z$. Since $\Gamma \in \CBC$ we have that $\lim\limits_{-\frown e} H_*(\Gamma / \langle g \rangle;\Q) = 0$, so we compute
\begin{align*}
\big(\colim_{- \smile e} H^*(\Gamma / \langle g \rangle;\Q)\big)^* &\cong \lim_{- \smile e} \big(H^*(\Gamma / \langle g \rangle;\Q)\big)^* \\
                                                                         &\cong \lim_{- \frown e} H_*(\Gamma / \langle g \rangle;\Q) \\
                                                                         & = 0
\end{align*}
where the second isomorphism uses that $H_*(\Gamma/ \langle g \rangle;\Q)$ has finite type. Restricting to even degrees we find that 
\[0 = \colim\limits_{- \smile e} H^{2*}(\Gamma / \langle g \rangle;\Q) \cong H^{2*}(\Gamma / \langle g \rangle;\Q)\left[\tfrac{1}{e}\right]\]
which implies that $e \in H^2(\Gamma / \langle g \rangle;\Q)$ is nilpotent, and so $\rho^* \colon H^*(\B^2\Z;\Q) \to H^*(\B(\Gamma/\mathcal \langle g\rangle);\Q)$ is not injective, as required.
\end{proof}

\begin{Prop}\label{prop:IKPimpliesIVP}
If an oriented rational Poincar\'e duality group lies in $\IKP$ then it lies in $\IVP$. The same statement holds for $\KC$ and $\GC$.
\end{Prop}
\begin{proof}
Suppose that $\Gamma \in (\mathcal{I})\KC$. Let $C$ denote either $\langle g \rangle$ for $g$ a central element of infinite order (in the case of $\IKP$) or the entire centre of $\Gamma$ (in the case of $\KC$; the centre must be non-trivial if $\Gamma \in \KC$).

Let $0 \neq x \in H^*(\B^2 C;\Q)$ be such that $\pi^*(x) = 0$, and $y \in H^*(\B(\Gamma/C);\Q)$ be arbitrary. Then by the projection formula we have 
\[ 0 = \pi_!(\pi^*(x) \smile y) = x \smile \pi_!(y) \in H^*(\B^2 C;\Q).\]
Since this ring is a domain it follows that $\pi_!(y) = 0$ and since $y$ is arbitrary we have $\pi_! = 0$, so that $\Gamma \in (\mathcal{I})\GC$.
\end{proof}

The following proposition is a converse to the first case in the above proposition. We have not been able to prove the converse to the second case. However, we do not use either result, and include it only for the sake of interest.

\begin{Prop}\label{prop:IVPimpliesIKP}
If a group lies in $\IVP$ then it lies in $\IKP$. 
\end{Prop}
\begin{proof}
We will show the contrapositive, so suppose $\Gamma$ does not lie in $\IKP$. Then it has a central element $g$ of infinite order such that in the corresponding fibration
$$\B\Gamma \lto \B(\Gamma/\langle g \rangle) \overset{\rho}\lto \B^2 \Z $$
the map $\rho^*$ is injective. Choose an $n$ such that $2n \gg d$, and consider restricting the fibration $\rho$ to the skeleton $\mathbb{CP}^n \subset \mathbb{CP}^\infty = \B^2 \Z$, to give a fibration
$$\B\Gamma \lto E \overset{\nu}\lto \mathbb{CP}^n.$$
By considering the induced map of Serre spectral sequences, we see that $\nu^*$ is also injective. 

In the homological Serre spectral sequence for $\nu$, the product of fundamental classes $[\mathbb{CP}^n] \otimes [\B\Gamma] \in H_{2n}(\mathbb{CP}^n;\Q) \otimes H_d(\B\Gamma;\Q) = E^2_{2n,d}$ is a permanent cycle, as there is no space for differentials. By Poincar{\'e} duality on $\mathbb{CP}^n$ and $\B\Gamma$, capping with this class gives an isomorphism of spectral sequences
$$- \frown [\mathbb{CP}^n] \otimes [\B\Gamma] : E^{p,q}_r \lto E^r_{2n-p, d-q}$$
between the cohomological and homological Serre spectral sequences.

Now as $\nu^*$ is injective, $\nu_*$ is surjective: in other words, there are no differentials exiting the bottom row of the homological Serre spectral sequence, in particular exiting the group $E^r_{2n, 0}$. But then by the duality isomorphism there are no differentials exiting the group $E_r^{0, d}$, and hence $\nu_! : H^d(E;\Q) \to H^0(\mathbb{CP}^n; \Q)$ is onto. Choose $u \in H^d(E;\Q)$ such that $\nu_!(u)=1$.

As $E \to \B(\Gamma/\langle g \rangle)$ is $2n$-connected, and $2n \gg d$, the class $u$ extends to a class of the same name in $H^d(\B(\Gamma/\langle g \rangle);\Q)$, which then satisfies $\rho_!(u)=1$. Thus $\Gamma$ is not in $\IVP$.
\end{proof}

Putting together \cref{Markus' trick} and \cref{prop:IKPimpliesIVP} we find:

\begin{Cor}\label{mthm}
Let $M$ be an oriented aspherical manifold with fundamental group $\Gamma$. If $\Gamma$ has non-trivial centre and satisfies the Burghelea conjecture and the {\pbc} with $\Q$-coefficients then 
\[0 = \kappa_c(M) \in H^*(\BTOPtw_h(M);\Q)\]
for all $c \in H^*(\BSTOP \times K(\Z,d);\Q)$.
\end{Cor}



The final condition we will discuss is as follows.

\begin{Def}
Let $\CP$ denote the class of groups $\Gamma$ with $\cd_\mathbb Q(\Gamma) < \infty$ and $\cd_\mathbb Q^{\tr}(\Gamma/\C(\Gamma)) < \infty$.
\end{Def}

\begin{Prop}\label{CP implies IKP}
If $\Gamma \in \CP$ then $\Gamma \in \IKP$, and if in addition $C(\Gamma) \otimes \Q \neq 0$ then $\Gamma \in \KC$.
\end{Prop}
\begin{proof}
Suppose that $\Gamma \in \CP$ and $g \in \Gamma$ is a central element of infinite order, with corresponding fibration
$$\B\Gamma \lto \B(\Gamma/\langle g \rangle) \overset{\rho}\lto \B^2 \Z.$$
Writing $H^*(\B^2\Z;\Q)=\Q[\iota_2]$, the class $\rho^*(\iota_2)$ is by definition the
Euler class of the central extension
\[1 \longrightarrow \mathbb Z \stackrel{g}{\longrightarrow} \Gamma \longrightarrow \Gamma/\langle g \rangle \longrightarrow 1,\]
so the non-injectivity of $\rho^*$ is equivalent to the nilpotence of this Euler class. 

As $\Gamma$ has finite $\Q$-cohomological dimension so does any subgroup, in particular $C(\Gamma)$. Now $\cd^{\tr}_\Q(A) = \mathrm{rk}(A)$ for any abelian group $A$ 
and therefore $\cd^\tr_\Q(\C(\Gamma)/\langle g \rangle) \leq \cd_\Q^{\tr}(\C(\Gamma))$. 
We also have $\cd_\mathbb Q^{\tr}(\Gamma/\C(\Gamma)) < \infty$, so the Serre spectral sequence for the central extension
$$1 \lra C(\Gamma)/\langle g \rangle \lra \Gamma/\langle g \rangle \lra \Gamma/C(\Gamma) \lra 1$$
has finitely many non-zero rows and columns, and hence $\cd_\mathbb Q^{\tr}(\Gamma/\langle g \rangle) < \infty$. But then the Euler class $e \in H^2(\Gamma/\langle g \rangle;\Q)$ must be nilpotent.

For the final statement, if $C(\Gamma) \otimes \Q \neq 0$ then ${H}^{2*}(\B^2 C(\Gamma);\Q) = \mathrm{Sym}^{*}_\Q(\mathrm{Hom}(C(\Gamma),\Q))$ contains non-nilpotent elements of strictly positive degree, so as $\cd_\mathbb Q^{\tr}(\Gamma/\C(\Gamma)) < \infty$ it follows that the map $\pi^* \colon H^*(\B^2\C(\Gamma);\Q) \to H^*(\B(\Gamma/\C(\Gamma));\Q)$ cannot be injective.
\end{proof}

We do not know of an aspherical manifold whose fundamental group is not contained in $\CP$ (we pose as a question in \cref{questions} whether this is generally the case). One might in fact guess that $\B(\Gamma/\C(\Gamma))$ is a Poincar\'e complex whenever $\Gamma$ is the fundamental group of an aspherical manifold. 
This is true if both $\C(\Gamma)$ and $\Gamma/\C(\Gamma)$ are of type $F$, by the 2-out-of-3 property for Poincar\'e spaces for fibrations of \emph{finite} complexes \cite{Gottlieb2}. But the group $\Gamma/\C(\Gamma)$ is not even torsion-free in general: For example for $M$ one of the manifolds constructed in \cite{CWY} as counterexamples to a conjecture about free $S^1$-actions on aspherical manifolds with non-trivial centre, $\Gamma/\C(\Gamma)$ contains a non-trivial element of order $2$, and thus cannot even admit a finite dimensional model of its classifying space. 


Finally, we have the following convenient criterion for being in $\KC$:
\begin{Lemma}\label{lemma vanishing}
Let $\Gamma$ be a group with $\cd_\Q(\Gamma) < \infty$. If the map $\C(\Gamma) \to \Gamma^\ab\otimes\Q$ is non-trivial, then
\[\pi^* \colon H^2(\B^2\C(\Gamma);\Q) \lto H^2(\B(\Gamma/\C(\Gamma));\Q)\] 
is not injective, in particular $\Gamma \in \KC$.
\end{Lemma}
\begin{proof}
Suppose that $\pi^*$ is injective and consider the Serre spectral sequence for the fibration
\[ \B\C(\Gamma) \lto \B \Gamma \stackrel{q}{\lto} \B(\Gamma/\C(\Gamma)).\]
	Let $r$ be the rank of $\C(\Gamma)$, which is bounded by $\cd_\Q(\Gamma)$. Clearly the image of $\pi^*$ is contained in $\ker(q^*)$, which therefore by assumption contains an $r$-dimensional subspace. On the other hand $\ker(q^*)$ is the image of the differential 
\[ d_2 \colon H^1(\B\C(\Gamma);\Q) \lto H^2(\B(\Gamma/\C(\Gamma));\Q).\]
Since $\dim_\Q(H^1(\B\C(\Gamma);\Q)) = r$ it follows that this differential is an isomorphism. Therefore $H^1(\B\Gamma;\Q) \to H^1(\B(\C(\Gamma));\Q)$ is the zero map and the lemma follows by dualising.
\end{proof}

\subsection{Integral results}\label{intres}

Many of the above results are true integrally under an additional assumption, namely that $C(\Gamma)$ is finitely generated. It does not seem to be known whether this holds when $\Gamma$ is the fundamental group of an aspherical manifold (we pose this as a question in \cref{questions}).

\begin{Thm}\label{integral result}
If $M$ is an oriented aspherical manifold, which satisfies the Burghelea conjecture and the {\pbc} with $\Z$-coefficients and whose centre is non-trivial and finitely generated, then
\[0 = \kappa_c(M) \in H^*(\BTOPtw_h(M); \mathbb Z)\]
for all $c \in H^*(\BSTOP \times K(\Z,d); \mathbb Z)$.
\end{Thm}
\begin{proof}
The group $\Gamma = \pi_1(M)$ has type $F$ and lies in $\CBC$ and so by \cref{Markus' trick} it lies in $\IKP$. 
Thus the homomorphism
\[\rho(g)^*\colon H^*(\B^2\mathbb Z; \mathbb Z) \stackrel{}\longrightarrow H^*(\B(\Gamma/\langle g \rangle); \mathbb Z)\]
on integral cohomology has non-trivial kernel: For both sides the rational cohomology is the rationalisation of the integral cohomology, since all the homology groups are finitely generated (for the left hand side this is immediate and for the right hand side it follows from the Serre spectral sequence and the fact that $\Gamma$ has type $F$, so $F_\infty$). As $\Gamma $ lies in $\IKP$ there is an element $0 \neq x \in H^*(\B^2 \mathbb Z; \mathbb Z)$ such that $\pi^*(x)$ is torsion in $H^*(\B(\Gamma/\langle g \rangle); \mathbb Z)$, whence an appropriate multiple of $x$ gives a non-zero element in the kernel of $\pi^*$ since $H^*(\B^2 \mathbb Z; \mathbb Z)$ is torsion-free.

\cref{prop:IKPimpliesIVP} remains valid integrally. More specifically, a Poincar\'e duality group $\Gamma$ that lies in $\IKP$ also lies in the obvious integral version of $\IVP$. The argument only used that $H^*(\B^2 \mathbb Z)$ is a domain, which holds for both $\mathbb Z$ and $\mathbb Q$ coefficients.

On the other hand the argument of \cref{Oscar's trick} 
uses that an element in $H^*(\B^2\C(\Gamma))$ which vanishes under the restriction to $H^*(\B^2\mathbb Z)$ for all embeddings $\mathbb Z \to C(\Gamma)$ has to be trivial. While this is true with rational coefficients for an arbitrary abelian group of finite rank (which the centre of a rational Poincar\'e duality group always is), if $C(\Gamma)$ is not finitely generated then we cannot, for example, exclude the possibility $C(\Gamma)=\Q$, in which case $H^3(\B^2 C(\Gamma); \mathbb Z) = \Ext(\mathbb Q, \mathbb Z) \neq 0$, and such classes vanish under any embedding $\mathbb Z \to C(\Gamma)$. However (because $C(\Gamma)$ is torsion-free) under the assumption that $C(\Gamma)$ is finitely-generated we have $H^{2n}(\B^2 C(\Gamma);\Z) = \Sym^n(\Hom(C(\Gamma),\Z))$ and the argument of  \cref{Oscar's trick} goes through.  
\end{proof}

The integral vanishing of all tautological classes for smooth bundles is not implied by this result, as not every Pontryagin class lies in the image of the forgetful map $H^*(\BSTOP; \mathbb Z) \to H^*(\BSO; \mathbb Z)$. However the work of Kirby--Siebenmann \cite[p.\ 200]{KS} describes the fibre of the map $\BSO \to \BSTOP$ in terms of groups of homotopy spheres. Since these are finite, some multiple of every class in $H^*(\BSO(d);\Z)$ is pulled back from $\BSTOP$ and in principle these multiples can be determined in terms of the orders of the groups of homotopy spheres; we shall refrain from spelling this out. In particular, \cref{integral result} implies that there is a bound on the order of $\kappa_c(M) \in H^*(\B\DIFF_h(M);\Z)$ independent of $M$.

\section{Examples and questions}\label{sec6}

In this section we will discuss explicit examples of manifolds whose tautological classes vanish and in particular satisfy our conjecture. We also provide counterexamples to a few possible extensions. For the reader's convenience let us first recall our conjecture.

\begin{conj*}
Let $M$ be a closed, connected, oriented, aspherical manifold. If $\C(\pi_1(M)) \neq 0$ then 
\[0 = \kappa_c \in H^*(\B\TOPtw_h(M);R)\]
for all $c \in H^*(\BSTOP\times K(\Z,d);R)$.
\end{conj*}

\subsection{Examples} 
We want to start out with some rather abstract examples that satisfy our conjecture with $R=\Q$.

\begin{Thm}\label{abstract examples}
The following classes of groups satisfy Burghelea's conjecture and are therefore contained in $\IKP$:
\begin{enumerate}
\item[(i)] Cocompact lattices in almost connected Lie groups,
\item[(ii)] $\CAT(0)$-groups,
\item[(iii)] solvable groups, linear groups over $\Q$,
\item[(iv)] groups of polynomial growth, hyperbolic groups, arithmetic groups, and
\item[(v)] elementary amenable groups.
\end{enumerate}
In particular, if $M$ is an oriented aspherical manifold with fundamental group in one of the above classes then $M$ satisfies our conjecture with $\Q$-coefficients.
\end{Thm}
\begin{proof}
If the dimension of $M$ is smaller than 4, then as explained \cref{vanishing low dimensions} the vanishing of rational tautological classes is known anyhow. We first claim that all groups in the above list are contained in the class $\LFJ$, so by \cref{pbc2} any such $M$ satisfies the {\pbc} with $\Z[\tfrac{1}{2}]$-coefficients provided the dimension of $M$ is at least 4. (In fact, with the exception of elementary amenable groups, these groups are contained in the class $\calfj$ and so by \cref{FJ vs BBC} any such $M$ satisfy the {\pbc} with any coefficients.) 
The groups not covered by \cref{status FJ conjecture} are groups of polynomial growth which are virtually solvable (in fact virtually nilpotent) by a celebrated theorem of Gromov \cite{Gromov2} and hence lie in $\calfj$, and linear groups over $\Q$ which lie in $\calfj$ by \cite{R}.

Hence it suffices to verify that all above groups satisfy Burghelea's conjecture. This has been done in the following references: (i) is dealt with in \cite[Theorem 4.27]{EM}. For (ii) see \cite[Corollary 4.8]{EM}, (iii) is \cite[Theorems 2.3 \& 2.4]{Eckmann}, (iv) is \cite[Theorem 4.3]{Ji} and (v) is \cite[Theorem 4.20]{EM} as finite Hirsch length is equivalent to finite homological dimension by \cite[Theorem I.2]{BK}.
\end{proof}

Actually, in the case of $\CAT(0)$-groups we do have an integral result.
\begin{Prop}\label{integral CAT(0)}
If $M$ is an orientable aspherical manifold of dimension at least 4 whose fundamental group is $\CAT(0)$, then $M$ satisfies our conjecture with $\Z$-coefficients.
\end{Prop}
\begin{proof}
$\CAT(0)$-groups are semihyperbolic, see \cite[Corollary 4.8]{BH} and hence have finitely generated centre, see \cite[Proposition 4.15; (3)]{BH}. Thus \cref{integral result} applies.
\end{proof}

\begin{Rmk}
Even though we consider Burghelea's conjecture the bottleneck of our work, rather than the Farrell--Jones conjectures, there do exist groups for which Burghelea's conjecture is known
and the Farrell--Jones conjectures are not, e.g.\ linear groups over arbitrary fields of characteristic $0$ \cite[Theorem 2.4]{Eckmann}.
\end{Rmk}

Concretely, we obtain the following consequences.

\begin{Cor}\label{concrete examples}
Our conjecture holds 
\begin{enumerate}
\item[(i)] rationally for oriented aspherical manifolds of the form $\Gamma \backslash G/ K$, where $G$ is a connected Lie group, $K$ is a maximal compact subgroup and $\Gamma$ is a cocompact lattice in $G$, and
\item[(ii)] integrally for oriented aspherical manifolds admitting a metric of non-positive sectional curvature.
\end{enumerate}
\end{Cor}

Let us remind the reader that we already obtained stronger results for tori in \cref{npt}.

\begin{proof}
(i) is a direct consequence of \cref{abstract examples}. (ii) is a special case of \cref{integral CAT(0)} and thus holds even integrally, because the fundamental group of such manifolds are $\CAT(0)$. We are not aware of results about the finite generation of the centre of cocompact lattices in Lie groups, hence (i) is only a rational result. 
\end{proof}

As indicated all the groups appearing in \cref{abstract examples} are contained in $\IKP$. In addition to these examples we have the following result.

\begin{Prop}\label{closure props of KC}
The class $\IKP$ has the following properties:
\begin{enumerate}
\item[(i)] It is closed under extensions, and
\item[(ii)] if a group $\Gamma$ has a finite index subgroup $K$ which lies in $\IKP$, then $\Gamma$ lies in $\IKP$ as well.
\end{enumerate}

\end{Prop}

This purely group theoretic statement has the following geometric interpretation: Together with \cref{LFJ} and \cref{pbc2} item (i) verifies our conjecture rationally for total spaces of fibre bundles provided the fundamental groups of both base and fibre are in both $\IKP$ and $\LFJ$. So, for instance, by our previously established results, our conjecture holds for an arbitrary torus bundle over a non-positively curved manifold, or vice versa, and iterates of those. Similarly, part (ii) enables passage from the total space of a finite cover to the base.

\begin{proof}[Proof of \cref{closure props of KC}]

For (i), we consider a short exact sequence of groups
\[1 \lto K \lto \Gamma \stackrel{q}{\lto} Q \lto 1\]
and assume that $K$ and $Q$ are contained in the class $\IKP$. Let $g \in \Gamma$ be a central element of infinite order. We need to show that the map $\B(\Gamma/\langle g \rangle) \xto{\rho(g)} \B^2\Z$
classifying the central extension
\[ 1 \lto \Z \overset{g}\lto \Gamma \lto \Gamma/\langle g \rangle \lto 1\]
has a non-trivial kernel in rational cohomology. We distinguish two cases, namely whether or not the central element $q(g) \in Q$ has infinite order.
We begin with the case where $q(g)$ is of infinite order in $Q$. If this is the case we have a commutative diagram
\[\xymatrix{\B(\Gamma/\langle g \rangle) \ar[r] \ar[d] & \B^2\Z \ar@{=}[d] \\ \B(Q/\langle q(g) \rangle) \ar[r] & \B^2 \Z.}\]
Since $Q$ lies in the class $\IKP$ it follows that the lower horizontal map has a non-trivial kernel in rational cohomology. By the commutativity of the diagram the same follows for the upper horizontal map.

Now let us assume that $q(g)$ is of finite order, say $n$, and observe that this makes $g^n$ a central element of $K$. Consider the short exact sequence of groups
\[ 1 \lto K/\langle g^n \rangle \lto \Gamma/\langle g^n\rangle \lto \Gamma/K \lto 1.\]
The central extension $1 \to \langle g^n \rangle \to K \to K/\langle g^n \rangle \to 1$ is pulled back from $1 \to \langle g^n \rangle \to \Gamma \to \Gamma/\langle g^n \rangle \to 1$ along $K/\langle g^n \rangle \to \Gamma/\langle g^n\rangle$, so it is classified by the composition
$$\tau : \B(K/\langle g^n \rangle) \overset{i}\lto \B(\Gamma/\langle g^n \rangle) \overset{\rho(g^n)}\lto \B^2 \Z.$$
Since $K \in \IKP$ there is an element
\[0 \neq x \in \ker(\tau^* \colon H^*(\B^2 \mathbb Z; \mathbb Q) \lra H^*(\B(K/\langle g^n \rangle); \mathbb Q)),\]
using which we define $y := \rho(g^n)^*(x)$ so that $i^*(y)=0$. Thus, in the Serre spectral sequence for this short exact sequence of groups, the class $y$ has Serre filtration $\geq 1$. Now $\Gamma/K \cong Q$ has finite rational cohomological dimension, say $d$, but then $y^{d+1}$ has Serre filtration $\geq d+1$ and so is zero. Thus we deduce that
\[0 \neq x^{d+1} \in \ker(\rho(g^n)^*\colon H^*(\B^2 \mathbb Z; \mathbb Q) \lra H^*(\B(\Gamma/\langle g^n \rangle); \mathbb Q)).\]
To finish the proof of (ii) it hence suffices to verify the following
\begin{claim}
If the map $\rho(g^n)$ has a non-trivial kernel on rational cohomology for some non-zero integer $n$, then so does the map $\rho(g)$.
\end{claim}
\begin{proof}[Proof of Claim]
We consider the sequence of subgroups
$\langle g^n \rangle \subseteq \langle g \rangle \subseteq \Gamma$
and conclude that the map
\[ \B(\Gamma/\langle g^n \rangle) \lto \B(\Gamma/\langle g \rangle) \]
is a rational equivalence as its fibre $\B(\Z/n)$ is rationally contractible. From the commutative diagram
\[\xymatrix{\B(\Gamma/\langle g^n \rangle) \ar[r] \ar[d] & \B^2\Z \ar[d]^-{\cdot n} \\ \B(\Gamma/\langle g \rangle) \ar[r] & \B^2\Z }\]
and the fact that the right vertical map induces an isomorphism in rational cohomology we conclude the claim.
\end{proof}

To obtain (ii), let $K \subseteq \Gamma$ be a finite index subgroup with $K \in \IKP$ and $g \in \Gamma$ a central element of infinite order. Since $K$ has finite index in $\Gamma$ it follows that there exists an $n$ such that $g^n$ lies in $K$, and thus is a central element of infinite order in $K$. Therefore by assumption the composite
\[\rho(g^n)^* \colon H^*(\B^2\mathbb Z; \mathbb Q) \longrightarrow H^*(\B(\Gamma/\langle g^n \rangle);\mathbb Q) \longrightarrow H^*(\B(K/\langle g^n \rangle);\mathbb Q) \]
has non-trivial kernel. Now, the map $\B K \to \B\Gamma$ has finite, discrete homotopy fibres, so the second map in the composite is injective. It thus follows that $\rho(g^n)$ has non-trivial kernel on rational cohomology, so by the above (verified) claim, so does $\rho(g)$. Hence $\Gamma$ lies in $\IKP$.
\end{proof}

\begin{Rmk}
Let $M$ be an aspherical manifold of dimension at least $5$ satisfying the block Borel conjecture and $\Gamma$ be its fundamental group. Suppose the conjecture is true rationally for $M$ and that $\Out(\Gamma)$ has finite rational cohomological dimension. Then it follows easily from the Serre spectral sequence for the fibration
\[ \B\wt{\TOP}_h(M) \lto \B\wt{\TOP}(M) \lto \B\Out(\Gamma) \]
that the tautological classes in the rational cohomology of $\B\wt{\TOP}(M)$ are nilpotent.

This is for instance the case if the fundamental group of $M$ is nilpotent: Since nilpotent groups are solvable, we know that our conjecture is satisfied. Moreover, a finitely generated nilpotent group is polycyclic, hence by \cite[Theorem 1.1]{BG} its outer automorphism group is arithmetic and thus has finite rational cohomological dimension, see \cite{Borel}.
\end{Rmk}

We provide two more examples of manifolds satisfying our conjecture using \cref{lemma vanishing}.
\begin{Prop}
Let $N$ be an oriented aspherical manifold and assume that $\pi_1(N)$ is a Farrell--Jones group. Let $e \in H^2(N;\Z)$ be a torsion class,
and let $M$ be the total space of the principal $S^1$-bundle classified by $e$. Then 
our conjecture holds rationally for $M$.
\end{Prop}
\begin{proof}
We first prove that $M$ satisfies the block Borel conjecture.
For this we observe that $M$ is finitely covered by the trivial $S^1$-bundle over $N$, as $e$ is a torsion class. Thus $\pi_1(M)$ contains $\pi_1(N)\times \Z$ as a finite index subgroup and hence is a Farrell--Jones group.
To prove the proposition we consider the Serre spectral sequence for the fibration
\[ S^1 \lto M \lto N.\]
By inspection, the map $H_1(S^1;\Q) \to H_1(M;\Q)$ is non-zero, since $e$ is rationally zero. We deduce the proposition from \cref{lemma vanishing} using the centrality of the inclusion $\pi_1(S^1) \to \pi_1(M)$.
\end{proof}

\begin{Cor}
Let $N$ be an closed, oriented, aspherical manifold such that $\pi_1(N)$ is a Farrell--Jones group. Then our conjecture holds rationally for $M = N \times S^1$.
\end{Cor}

Let us close this section by considering mapping tori. Let $\varphi\colon M \to M$ be a orientation preserving homeomorphism of an oriented aspherical manifold $M$. Then $\varphi$ determines an outer automorphism $\varphi_\ast$ of $\pi_1(M)$. Picking a representing automorphism $\hat\varphi\colon \pi_1(M) \to \pi_1(M)$, there is an isomorphism between the fundamental group of the mapping torus $M_\varphi$, which is again oriented and aspherical, and $\pi_1(M) \rtimes_{\hat\varphi} \mathbb Z$. The centre of such a semi-direct product is readily computed to be
\[\{(g,n) \mid \hat\varphi^n = c_{g^{-1}}, \hat\varphi(g) = g\}\]
and so using \cref{lemma vanishing} we obtain:

\begin{Prop}\label{torunks}
Let $\varphi\colon M \to M$ be an automorphism of an oriented, aspherical manifold $M$, such that $\pi_1(M)$ is a Farrell--Jones group. Assume that the induced automorphism on $\pi_1(M)$ admits a representative $\hat\varphi$ such that $\hat\varphi^n$ is conjugation by an element of $\pi_1(M)$ fixed by $\hat\varphi$ for some $n > 0$. Then our conjecture holds rationally for $M_\varphi$.
\end{Prop}

To complete the proof note that the {\pbc} with $\Q$-coefficients holds for $\pi_1(M_\varphi) \cong \pi_1(M) \rtimes_{\hat\varphi} \mathbb Z$ by \cref{LFJ} and \cref{pbc2}. The assumption is in particular satisfied for a finite order automorphism with a fixed point, e.g.\ the identity, thereby giving another proof that $S^1 \times M$ satisfies our conjecture rationally, whenever $\pi_1(M)$ is a Farrell--Jones group. We leave it as an exercise to the reader to check that the condition given on $\varphi$ is actually independent of the representative $\hat\varphi$ chosen.

Finally, let us again consider the manifolds constructed in \cite{CWY}, that provide counterexamples to $\B(\Gamma/C(\Gamma))$ being a Poincar\'e complex. While we do not know whether the fundamental groups of these manifolds are Farrell--Jones groups, we still have:

\begin{Prop}
The aspherical manifolds constructed by Cappell--Weinberger--Yan just mentioned have fundamental groups in $\IKP$.
\end{Prop}

\begin{proof}
There exist 2-fold covers of these manifolds which are of the form $S^1 \times V$, where $V$ is an aspherical manifold with centreless fundamental group, see \cite{CWY}, essentially by construction. The group $\pi_1(S^1 \times V)$ lies in the class $\IKP$, because its centre is a direct factor. \cref{closure props of KC} part (ii) then yields the claim.
\end{proof}

Finally, we record another consequence of our vanishing results.
\begin{Prop}
Let $B$ be a stably parallelisable manifold of dimension $n>0$ and let $M$ be an oriented aspherical manifold whose fundamental group is a Farrell--Jones group and satisfies the central part of Burghelea's conjecture. Suppose $\pi\colon E \to B$ is an $M$-manifold bundle with trivial fibre transport. Then all Pontryagin numbers of $E$ vanish.
\end{Prop}
\begin{proof}
Since $B$ is stably parallelisable, we find that the stable vertical tangent bundle $T^s_v(\pi)$ is isomorphic to the stable tangent bundle $T^s(E)$ of $E$. By our main theorem, we obtain that $0 = \kappa_{c}(\pi) \in H^n(B;\Q)$, for any $c \in H^{n+d}(\BSTOP;\Q)$. Since $\pi_! \colon H^{n+d}(E;\Q) \to H^n(B;\Q)$ is an isomorphism (where $d$ is the dimension of $M$), we deduce the proposition from
\[ 0 = \kappa_c(\pi) = \pi_!(c(T^s_v(\pi))) = \pi_!(c(T^s(E))).\]
\end{proof}

\subsection{Several open questions}\label{questions}
We want to finish with some open problems that would be interesting to address. The first question seems to be a known open problem:

\begin{Question}
Let $M$ be a closed, connected, aspherical manifold. Is the centre of its fundamental group finitely generated?
\end{Question}
Naturally, this is obvious for centreless, e.g.\ hyperbolic, groups, but it is also true for torsion-free nilpotent groups, which (unless trivial themselves) always have non-trivial centre, see \cite[Proposition 6.19]{Neofytidis}. Furthermore, as mentioned earlier, $\CAT(0)$-groups have finitely generated centre.

\begin{Question}
Are fundamental groups of aspherical manifolds always contained in $\CP$?
\end{Question}

This is true for several classes of manifolds (e.g.\ those with nilpotent fundamental group, since $\Gamma/C(\Gamma)$ is then finitely generated, nilpotent and torsion-free (by \cite[5.2.19]{Robinson}) and therefore poly-$\Z$ by \cite[5.2.20]{Robinson}), and would yield a general proof of the rational part of our conjecture. The statement does not seem logically comparable to Burghelea's conjecture, but could be easier to prove. 

A positive answer to the next question would provide a geometric reason for $\Gamma \in \IKP$.

\begin{Question}
Let $M$ be an oriented, aspherical manifold with fundamental group $\Gamma$ and let $g \in \Gamma$ be a central element. Is there a finite cover $N \to M$ such that $g \in \pi_1(N)$ and such that $g$ is realised by a principal $S^1$-action on $N$?
\end{Question}

Note that the passage to a finite cover really is necessary. The examples from \cite{CWY} are aspherical manifolds $M$ such that the quotient $\Gamma/\C(\Gamma)$ contains an element of order $2$. This question has a positive answer in the case of 3-manifolds, by the proof of the Seifert fiber space conjecture \cite{Gabai, CassonJungreis}.

\begin{Question}
Do the tautological classes of a smooth, oriented aspherical manifold $M$ vanish in the cohomology of $\B\DIFF_h(M)$ with arbitrary coefficients?
\end{Question}

\begin{Question}
Do the tautological classes vanish in the cohomology of $\B\TOPtw{\mathstrut}^{\>\!+}(M)$ if $M$ is oriented, aspherical, and either odd dimensional or its fundamental group has non-trivial centre?
\end{Question}
We expect the answer to both parts of the question to be no, but did not find a counterexample.

\begin{Question}
Suppose $M$ is an oriented aspherical manifold whose fundamental group has a non-trivial centre. Is $M$ nullbordant? 
\end{Question}

If $M$ satisfies our conjecture rationally, then it follows that $M$ is a torsion element in the bordism ring. Notice that even if $M$ is smooth and satisfies the integral version of our conjecture we cannot yet deduce that $M$ is nullbordant. Really, we are asking about Stiefel--Whitney numbers of an aspherical manifold and how their triviality relies on the non-triviality of the centre of its fundamental group.

\bibliographystyle{amsalpha}
\bibliography{HLLRW-asphericaltautological}

\end{document}